\documentclass[aos,preprint]{imsart}

\RequirePackage[OT1]{fontenc}
\RequirePackage{amsthm,amsmath}
\RequirePackage[square,numbers,sort]{natbib}
\RequirePackage[colorlinks,citecolor=blue,urlcolor=blue]{hyperref}

\RequirePackage{amsmath,amssymb,amsfonts,amsthm,latexsym}
\RequirePackage{amscd,graphicx,color,enumerate,float,caption}
\RequirePackage[all]{xy}
\RequirePackage{comment}

\arxiv{arXiv:1705.10735}
\startlocaldefs

\numberwithin{equation}{section}
\theoremstyle{plain}

\newtheorem{theorem}{Theorem}[section]
\newtheorem{lemma}[theorem]{Lemma}
\newtheorem{proposition}[theorem]{Proposition}
\newtheorem{corollary}[theorem]{Corollary}

\theoremstyle{definition}
\newtheorem{definition}[theorem]{Definition}

\newtheorem{remark}[theorem]{Remark}

\newcommand{\R}{\mathbb{R}}

\newcommand{\ttinf}{2\rightarrow\infty}

\newcommand{\rmF}{\textrm{F}}
\newcommand{\rmTr}{\textrm{tr}}

\newcommand{\rmRank}{\textrm{rank}}
\newcommand{\rmMax}{\textrm{max}}
\newcommand{\rmMin}{\textrm{min}}

\endlocaldefs
\begin{document}

\begin{frontmatter}
\title{The two-to-infinity norm and singular subspace geometry with applications to high-dimensional statistics\thanksref{T1}}
\runtitle{The two-to-infinity norm and singular subspace geometry}
\thankstext{T1}{This work is partially supported by the XDATA program of the Defense Advanced Research Projects Agency (DARPA) administered through Air Force Research Laboratory (AFRL) contract FA8750-12-2-0303 and by the DARPA D3M program administered through AFRL contract FA8750-17-2-0112. This work is also supported by the Acheson J.~Duncan Fund for the Advancement of Research in Statistics at Johns Hopkins University.}

\begin{aug}
\author{\fnms{Joshua}~\snm{Cape}\and\fnms{Minh}~\snm{Tang}\and\fnms{Carey~E.}~\snm{Priebe}
	\ead[label=e1]{joshua.cape@jhu.edu}
	\ead[label=e2]{minh@jhu.edu}
	\ead[label=e3]{cep@jhu.edu}
}

%\author{\fnms{Joshua} \snm{Cape}
	%\thanksref{t1,t2,m1}
%	\ead[label=e1]{joshua.cape@jhu.edu} \and}
%\author{\fnms{Minh} \snm{Tang}
	%\thanksref{t3,m1,m2}
%	\ead[label=e2]{mtang10@jhu.edu} \and}
%\author{\fnms{Carey E.} \snm{Priebe}
	%\thanksref{t1,m2}
%	\ead[label=e3]{cep@jhu.edu}
%}

%\thankstext{t1}{Some comment}
%\thankstext{t2}{First supporter of the project}
%\thankstext{t3}{Second supporter of the project}
\runauthor{J.~Cape, M.~Tang, and C.~E.~Priebe}

\affiliation{Johns Hopkins University
	%\thanksmark{m1} and Another University\thanksmark{m2}
	}

\address{Department of Applied Mathematics and Statistics\\
Johns Hopkins University\\
3400 North Charles Street\\
Baltimore, Maryland 21218\\
USA\\
\printead{e1}\\
\phantom{E-mail:\ }\printead*{e2}\\
\phantom{E-mail:\ }\printead*{e3}}
%\address{Address of the Third author\\
%Usually a few lines long\\
%Usually a few lines long\\
%\printead{e3}\\
%\printead{u1}}
\end{aug}

\begin{abstract}
	The singular value matrix decomposition plays a ubiquitous role throughout statistics and related fields. Myriad applications including clustering, classification, and dimensionality reduction involve studying and exploiting the geometric structure of singular values and singular vectors.
	
	This paper provides a novel collection of technical and theoretical tools for studying the geometry of singular subspaces using the two-to-infinity norm. Motivated by preliminary deterministic Procrustes analysis, we consider a general matrix perturbation setting in which we derive a new Procrustean matrix decomposition. Together with flexible machinery developed for the two-to-infinity norm, this allows us to conduct a refined analysis of the induced perturbation geometry with respect to the underlying singular vectors even in the presence of singular value multiplicity. Our analysis yields singular vector entrywise perturbation bounds for a range of popular matrix noise models, each of which has a meaningful associated statistical inference task. In addition, we demonstrate how the two-to-infinity norm is the preferred norm in certain statistical settings. Specific applications discussed in this paper include covariance estimation, singular subspace recovery, and multiple graph inference. 
	
	Both our Procrustean matrix decomposition and the technical machinery developed for the two-to-infinity norm may be of independent interest.
\end{abstract}



\end{frontmatter}

\section{Introduction}
\label{sec:Introduction}
%%%%%%%%%%%%%%%%%%%%%%%%%%%%%%%%%%%%%%%%%%%%%%%%%%%
\subsection{Background}
\label{subsec:Background}
The geometry of singular subspaces is of fundamental importance throughout a wide range of fields including statistics, machine learning, computer science, applied mathematics, and network science. Singular vectors (resp.~eigenvectors) together with their corresponding subspaces and singular values (resp.~eigenvalues) appear throughout various statistical applications including 
principal component analysis
\cite{nadler2008finite, jolliffe1986principal, Cai-et-al--2013AoS, koltchinskii2017pca},
covariance matrix estimation \cite{johnstone2001PCA, Fan-et-al--2013JRSSB, Fan-et-al--2015AoS, Fan-et-al--2016},
spectral clustering \cite{Rohe-Chatterjee-Yu--2011, Luxburg--2007, Lei-Rinaldo--2015, sarkar2015role},
and graph inference \cite{Tang-et-al--2014, Tang-et-al--2016, Tang-Priebe--2016+}.

Singular subspaces are also studied in random matrix theory, a discipline which has come to have a profound influence on the development of high-dimensional statistical theory \cite{Bai-Silverstein--2010, Yao-et-al--2015}. In random matrix theory, topics of interest include the phenomenon of eigenvector delocalization \cite{Rudelson-Vershynin--2015} as well as the spectral behavior of (often low rank) matrices undergoing random perturbations \cite{O'Rourke-Vu-Wang--2014, benaych2011eigenvalues}. For an overview of recent work on the properties of eigenvectors of random matrices, see \cite{O'Rourke-Vu-Wang--2016JCTA}. Further discussion of how random matrix theory has come to impact statistics can be found in \cite{Paul-Aue--2014}.

From a computational perspective, optimization algorithms in signal processing and compressed sensing are often concerned with the behavior of singular vectors and subspaces \cite{Eldar-Kutyniok--2012}. The study of algorithmic performance on manifolds and manifold learning, especially involving the Grassmann and Stiefel manifolds, motivates related interest in a collection of Procrustes-type problems \cite{Edelman-et-al--1998, Bojanczyk-Lutoborski--1999, Benidis-et-al--2016}. Procrustes analysis occupies an established area within the theoretical study of statistics on manifolds \cite{Chikuse--2012} and arises in applications including diffusion tensor imaging \cite{Dryden--2009} and shape analysis \cite{Dryden--2016}. See \cite{Gower-Dijksterhuis--2004} for an extended treatment of both theoretical and numerical aspects concerning Procrustes-type problems.

Foundational results from matrix theory concerning the perturbation of singular values, vectors, and subspaces date back to the original work of Weyl \cite{Weyl--1912}, Davis and Kahan \cite{Davis-Kahan--1970}, and Wedin \cite{Wedin--1972}, among others. Indeed, these results form the backbone of much of the linear algebraic machinery that has since been developed for use in statistics. The classical references \cite{Stewart-Sun--1990, Bhatia--1997, Horn-Johnson-I--2012} provide further treatment of these foundational results and historical developments.

%%%%%%%%%%%%%%%%%%%%%%%%%%%%%%%%%%%%%%%%%%%%%%%%%%%
\subsection{Overview}
\label{subsec:Overview}
This paper provides a collection of technical and theoretical tools for studying the perturbations of singular vectors and subspaces with respect to the two-to-infinity norm (defined below). Our main theoretical results are first presented quite generally and then followed by concrete consequences thereof to facilitate direct statistical applications. In this work, we prove perturbation theorems for both low and high rank matrices. Among the advantages of our methods is that we allow singular value multiplicity and merely leverage a population singular value gap assumption in the spirit of \cite{Yu-Wang-Samworth--2015}.

As a special case of our general framework and methods, we improve upon results in \cite{Fan-et-al--2016} wherein the authors obtain an $\ell_{\infty}$ norm perturbation bound for singular vectors of low rank matrices exhibiting specific coherence structure. In this way, beyond the stated theorems in this paper, our results can be applied analogously to robust covariance estimation involving heavy-tailed random variables.

Our Procrustes analysis complements the study of perturbation bounds for singular subspaces in \cite{Cai-Zhang--2016}. When considered in tandem, we demonstrate a Procrustean setting in which one recovers nearly rate-matching upper and lower bounds with respect to the two-to-infinity norm.

Yet another consequence of this work is that we extend and complement current spectral methodology for graph inference and embedding \cite{Tang-et-al--2016, Lyzinski-et-al--2014EJS,levin2017central}. To the best of our knowledge, we obtain among the first-ever estimation bounds for multiple graph inference in the presence of edge correlation.

%%%%%%%%%%%%%%%%%%%%%%%%%%%%%%%%%%%%%%%%%%%%%%%%%%%
\subsection{Setting}
\label{subsec:Setting}
This paper formulates and analyzes a general matrix decomposition for the aligned difference between real matrices $U$ and $\hat{U}$, each consisting of $r$ orthonormal columns (i.e.~partial isometries; Stiefel matrices; orthogonal $r$-frames), given by
\begin{equation}
	\label{eqn:Uhat-UW}
	\hat{U}-UW,
\end{equation}
where $W$ denotes an $r \times r$ orthogonal matrix. We focus on (but are strictly speaking not limited to) a certain ``nice'' choice of $W$ which corresponds to an ``optimal'' Procrustes transformation in a sense that will be made precise.

Along with the matrix decomposition considerations presented in this paper, we provide technical machinery for the two-to-infinity subordinate vector norm on matrices, defined for $A \in \R^{p_{1} \times p_{2}}$ by
\begin{equation}
	\|A\|_{2 \rightarrow \infty} := \underset{\|x\|_2=1}{\textrm{sup}}\|Ax\|_{\infty}.
\end{equation}
Together, these allow us to obtain a suite of perturbation bounds within an additive perturbation framework of the singular value decomposition.

The two-to-infinity norm yields finer uniform control on the entries of a matrix than the more common spectral and Frobenius norms. We shall demonstrate that, in certain settings, the two-to-infinity norm is preferable to these and to other norms. In particular, matrices exhibiting bounded coherence in the sense of \cite{Candes-Recht--2009} form a popular and widely-encountered class of matrices for which the two-to-infinity norm is demonstrably an excellent choice.

The two-to-infinity norm has previously appeared in the statistics literature, including in \cite{Lyzinski-et-al--2014EJS} wherein it is leveraged to prove that adjacency spectral embedding achieves perfect clustering for certain stochastic block model graphs. More recently, it has also appeared in the study of random matrices when a fraction of the matrix entries are modified \cite{Rebrova-Vershynin--2016}. In general, however, the two-to-infinity norm has received far less attention than other norms. Among the aims of this paper is to advocate for the more widespread consideration of the two-to-infinity norm.

%%%%%%%%%%%%%%%%%%%%%%%%%%%%%%%%%%%%%%%%%%%%%%%%%%%
\subsection{Sample application:~covariance estimation}
\label{subsec:Sample_application}
We pause here to present an application of our work and methods to estimating the top singular vectors of a structured covariance matrix.

Denote a random vector $Y$ by $Y:=(Y^{(1)},Y^{(2)},\dots,Y^{(d)})^{\top}\in\R^{d}$, and let $Y,Y_{1},Y_{2},\dots,Y_{n}$ be independent and identically distributed (i.i.d.) mean zero multivariate Gaussian random (column) vectors with common covariance matrix $\Gamma \in \R^{d \times d}$. Denote the singular value decomposition of $\Gamma$ by $\Gamma \equiv U \Sigma U^{\top} + U_{\perp} \Sigma_{\perp} U_{\perp}^{\top}$, where $[U|U_{\perp}]\equiv[u_{1}|u_{2}|\dots|u_{d}]\in\R^{d \times d}$ is an orthogonal matrix. The singular values of $\Gamma$ are indexed in non-increasing order, $\sigma_{1}(\Gamma)\ge\sigma_{2}(\Gamma)\ge\dots\ge\sigma_{d}(\Gamma)$, with $\Sigma:=\text{diag}(\sigma_{1}(\Gamma),\sigma_{2}(\Gamma),\dots,\sigma_{r}(\Gamma))\in\R^{r \times r}$, $\Sigma_{\perp}:=\text{diag}(\sigma_{r+1}(\Gamma),\sigma_{r+2}(\Gamma),\dots,\sigma_{d}(\Gamma))\in\R^{(d-r) \times (d-r)}$, $\delta_{r}(\Gamma):=\sigma_{r}(\Gamma)-\sigma_{r+1}(\Gamma)>0$, and $r \ll d$.  Here, $\Sigma$ may be viewed as containing the ``signal'' (i.e.~spike) singular values of interest, while $\Sigma_{\perp}$ contains the remaining ``noise'' (i.e.~bulk) singular values. The singular values of $\Gamma$ are not assumed to be distinct; rather, the assumption $\delta_{r}(\Gamma)>0$ simply specifies a singular value population gap between $\Sigma$ and $\Sigma_{\perp}$.

Let $\hat{\Gamma}_{n}$ denote the empirical covariance matrix $\hat{\Gamma}_{n}:= \frac{1}{n}\sum_{k=1}^{n}Y_{k}Y_{k}^{\top}$ with decomposition $\hat{\Gamma}_{n} \equiv \hat{U}\hat{\Sigma}\hat{U}^{\top} + \hat{U}_{\perp}\hat{\Sigma}_{\perp}\hat{U}_{\perp}^{\top}$.
Let $E_{n}:=\hat{\Gamma}_{n}-\Gamma$ denote the difference between the empirical and theoretical covariance matrices. Below, $C, c, c_{1}, c_{2},\dots$ are positive constants (possibly related) of minimal interest. Below, $\textnormal{Var}(Y^{(i)})$ denotes the variance of $Y^{(i)}$.

We are interested in the regime where the sample size $n$ and the covariance matrix dimension $d$ are simultaneously allowed to grow. In this regime, an important measure of complexity is given by the effective rank of $\Gamma$, defined as $\mathfrak{r}(\Gamma):=\textnormal{trace}(\Gamma)/\sigma_{1}(\Gamma)$ \cite{koltchinskii2017pca}.

\begin{theorem}[Application:~covariance estimation]
	\label{thrm:generalCovarEst}
	In Section~\ref{subsec:Sample_application}, assume that $\textnormal{max}\{\mathfrak{r}(\Gamma), \log d\} = o(n)$, $\sigma_{1}(\Gamma)/\sigma_{r}(\Gamma)\le c_{1}$, $\delta_{r}(\Gamma) \ge c_{2}\sigma_{r}(\Gamma) > 0$, and $\|U\|_{\ttinf} \le c_{3}\sqrt{r/d}$. Let $\nu(Y):=\textnormal{max}_{1 \le i \le d}\sqrt{\textnormal{Var}(Y^{(i)})}$. Then there exists an $r \times r$ orthogonal matrix $W$ and a constant $C>0$ such that with probability at least $1-d^{-2}$,
	\begin{align*}
		%\label{eq:generalCovarEst}
			\|\hat{U} - UW_{U}\|_{\ttinf}
			&\le C \sqrt{\tfrac{ \textnormal{max}\{\mathfrak{r}(\Gamma), \log d\}}{n}}
			\left(
			\tfrac{\nu(Y)r}{\sqrt{\sigma_{r}(\Gamma)}}
			+ \tfrac{\sigma_{r+1}(\Gamma)}{\sigma_{r}(\Gamma)} \right)\\
			&+ C \left(\tfrac{ \textnormal{max}\{\mathfrak{r}(\Gamma), \log d\}}{n}\right)
			\left(\sqrt{\tfrac{\sigma_{r+1}(\Gamma)}{\sigma_{r}(\Gamma)}}
			+ \sqrt{\tfrac{r}{d}}\right).
	\end{align*}
\end{theorem}
\begin{remark}
	\label{rem:covariance_theorem_context}
	In the setting of Theorem~\ref{thrm:generalCovarEst}, spectral norm probabilistic concentration \cite{koltchinskii2017concentration,koltchinskii2017pca} can be applied to yield a na\"{i}ve two-to-infinity norm bound of the form
	\begin{equation}
	\label{eq:covarianceNaiveSpectral}
		\|\hat{U} - UW\|_{\ttinf}
		\le C\sqrt{\tfrac{ \textrm{max}\{\mathfrak{r}(\Gamma), \log d\}}{n}}.
	\end{equation}
	When $\Gamma$ exhibits the additional spike structure $\Gamma \equiv U(\Lambda+c^{2}I)U^{\top} + c^{2}U_{\perp}U_{\perp}^{\top}$ with $\sigma_{1}(\Gamma) \ge c_{4}(d/r)$, then $\sqrt{\sigma_{r+1}(\Gamma)}, \nu(Y) \le c_{5} \sqrt{\sigma_{1}(\Gamma)}\sqrt{r/d}$, and so the bound in Theorem~\ref{thrm:generalCovarEst} simplifies to the form
	\begin{equation}
	\label{eq:covarianceSpikeSimplified}
		\|\hat{U} - UW\|_{\ttinf}
		\le C \sqrt{\tfrac{\textrm{max}\{\mathfrak{r}(\Gamma), \log d\}}{n}}\sqrt{\tfrac{r^{3}}{d}}.
	\end{equation}
	The bound in Eq.~(\ref{eq:covarianceSpikeSimplified}) manifestly improves upon Eq.~(\ref{eq:covarianceNaiveSpectral}) since here $r \ll d$ and $d$ is taken to be large.
\end{remark}

%%%%%%%%%%%%%%%%%%%%%%%%%%%%%%%%%%%%%%%%%%%%%%%%%%%
\subsection{Organization}
\label{subsec:Organization}
The rest of this paper is organized as follows.
Section~\ref{sec:Preliminaries} establishes notation, motivates the use of the two-to-infinity norm in a Procrustean context, and presents the additive matrix perturbation model considered in this paper.
Section~\ref{sec:Main_results} collects our main results which fall into two categories:~matrix decompositions and matrix perturbation theorems.
Section~\ref{sec:Applications} demonstrates how this paper improves upon and complements existing work in the literature by way of considering three statistical applications involving covariance estimation (see also Section~\ref{subsec:Sample_application}), singular subspace recovery, and multiple graph inference. Section~\ref{sec:Discussion} offers some final remarks.
Section~\ref{sec:Proofs} contains the technical machinery developed for this paper as well as proofs of our main theorems.

%%%%%%%%%%%%%%%%%%%%%%%%%%%%%%%%%%%%%%%%%%%%%%%%%%%
%%%%%%%%%%%%%%%%%%%%%%%%%%%%%%%%%%%%%%%%%%%%%%%%%%%
\section{Preliminaries}
\label{sec:Preliminaries}
%%%%%%%%%%%%%%%%%%%%%%%%%%%%%%%%%%%%%%%%%%%%%%%%%%%
\subsection{Notation}
\label{subsec:Notation}
All vectors and matrices in this paper are taken to be real-valued. The symbols $:=$ and $\equiv$ are used to assign definitions and to denote formal equivalence. For any positive integer $n$, let $[n]:=\{1,2,\dots,n\}$. We use $C_{\alpha}$ to denote a general constant that may change from line to line unless otherwise specified and that possibly depends only on $\alpha$ (either a parameter or an indexing value). Let $\mathcal{O}(\cdot)$ denote standard big-O notation, possibly with an underlying probabilistic qualifying statement. These conventions are simultaneously upheld when we write $\mathcal{O}_{\alpha}(\cdot)$. We let $\mathbb{O}_{p,r}$ denote the set of all $p \times r$ real matrices with orthonormal columns so that $\mathbb{O}_{p} \equiv \mathbb{O}_{p,p}$ denotes the set of orthogonal matrices in $\R^{p \times p}$.

For (column) vectors $x,y \in \R^{p_{1}}$ where $x \equiv (x_{1},\dots,x_{p_{1}})^{\top}$, the standard Euclidean inner product between $x$ and $y$ is denoted by $\langle x,y \rangle$. The classical $\ell_{p}$ vector norms are given by $\|x\|_{p} := \left(\sum_{i=1}^{p}|x_{i}|^{p}\right)^{1/p}$ for $1 \le p < \infty$, and $\|x\|_{\infty} := \text{max}_{i}|x_{i}|$. This paper also makes use of several standard matrix norms. Letting $\sigma_{i}(A)$ denote the $i$th largest singular value of $A$, then $\|A\|_{2}:=\sigma_{1}(A)$ denotes the spectral norm of $A$, $\|A\|_{\textrm{F}}:=\sqrt{\sum_{i}\sigma_{i}^{2}(A)}$ denotes the Frobenius norm of $A$, $\|A\|_{1}:=\text{max}_{j}\sum_{i}|a_{ij}|$ denotes the maximum absolute column sum of $A$, and $\|A\|_{\infty}:=\text{max}_{i}\sum_{j}|a_{ij}|$ denotes the maximum absolute row sum of $A$. Additionally, we consider $\|A\|_{\rmMax} := \rmMax_{i,j}|a_{ij}|$.

%%%%%%%%%%%%%%%%%%%%%%%%%%%%%%%%%%%%%%%%%%%%%%%%%%%
\subsection{Norm relations}
\label{subsec:Standard_norm_relations}
A central focus of this paper is on the two-to-infinity norm defined for matrices as $\|A\|_{\ttinf} := \textrm{sup}_{\|x\|_2=1}\|Ax\|_{\infty}$. Proposition~\ref{prop:2infty_as_max_Euclidean} establishes the elementary fact that this norm corresponds to the maximum Euclidean row norm of the matrix $A$, thereby making it easily interpretable and straightforward to compute. In certain settings, $\|\cdot\|_{\ttinf}$ will be shown to serve as an attractive surrogate for $\|\cdot\|_{\rmMax}$ in light of additional algebraic properties that $\|\cdot\|_{\ttinf}$ enjoys.

For $A\in\R^{p_{1} \times p_{2}}$, the standard relations between the $\ell_{p}$ norms for $p \in \{1,2,\infty\}$ permit quantitative comparison of $\|\cdot\|_{\ttinf}$ to the relative magnitudes of $\|\cdot\|_{\rmMax}$ and $\|\cdot\|_{\infty}$. In particular, these matrix norms are related through matrix column dimension via
\begin{equation*}
\label{eq:max_2infty_infty_inequal}
	\tfrac{1}{\sqrt{p_{2}}}
	\|A\|_{\ttinf}
	\le \|A\|_{\rmMax}
	\le \|A\|_{\ttinf}
	\le \|A\|_{\infty}
	\le \sqrt{p_{2}}\|A\|_{\ttinf}.
\end{equation*}
In contrast, the relationship between $\|\cdot\|_{\ttinf}$ and $\|\cdot\|_{2}$ depends on the matrix row dimension (Proposition~\ref{prop:2infty_&_spec}) via
\begin{equation*}
	\label{eq:2infty_spectral_inequal}
	\|A\|_{\ttinf}
	\le \|A\|_{2}
	\le \sqrt{p_{1}}\|A\|_{\ttinf}.
\end{equation*}

As an example, consider the rectangular matrix $A:=\{1/\sqrt{p_{2}}\}^{p_{1}\times p_{2}}$, for which $\|A\|_{\ttinf}=1$ while $\|A\|_{2}=\|A\|_{\textrm{F}}=\sqrt{p_{1}}$. This example, together with the above norm relations, demonstrates that possibly $\|A\|_{\ttinf} \ll \|A\|_{2}$ when the row dimension of $A$ is large relative to the column dimension, i.e.~$p_{1} \gg p_{2}$. Bounding $\|A\|_{\ttinf}$ would then be preferred to bounding $\|A\|_{2}$ when seeking more refined (e.g.~entrywise) control of $A$. The same observation holds trivially for the Frobenius norm which satisfies the well-known, rank-based relation with the spectral norm given by
\begin{equation*}
	\label{eqn:spectral_Frobenius_rank_inequal}
	\|A\|_{2} \le \|A\|_{\rmF} \le \sqrt{\text{rank}(A)}\|A\|_{2}.
\end{equation*}

We pause to point out that the two-to-infinity norm is not in general sub-multiplicative for matrices. Moreover, the ``constrained'' sub-multiplicative behavior of $\|\cdot\|_{\ttinf}$ (Proposition~\ref{prop:2infty_relations}), when taken together with the non-commutativity of matrix multiplication and standard properties of more common matrix norms, yields substantial flexibility when bounding matrix products and passing between norms. For this reason, a host of bounds beyond those presented in this paper follow naturally from the matrix decomposition results in Section~\ref{subsec:Matrix_decompositions}. The relative strength of derived bounds will depend upon underlying, application-specific properties and assumptions.

%%%%%%%%%%%%%%%%%%%%%%%%%%%%%%%%%%%%%%%%%%%%%%%%%%%
\subsection{Singular subspaces and Procrustes analysis}
\label{subsec:SSG&Procrustes}
Let $\mathcal{U}$ and $\hat{\mathcal{U}}$ denote the corresponding subspaces for which the columns of $U, \hat{U} \in \mathbb{O}_{p,r}$ form orthonormal bases, respectively. From the classical CS matrix decomposition, a natural measure of distance between these subspaces (corresp. matrices) is given via the canonical (i.e.~principal) angles between $\mathcal{U}$ and $\hat{\mathcal{U}}$. More specifically, for the singular values of $U^{\top}\hat{U}$, denoted by $\{\sigma_{i}(U^{\top}\hat{U})\}_{i=1}^{r}$ and indexed in non-increasing order, the canonical angles are the main diagonal elements of the $r \times r$ diagonal matrix
\begin{equation*}
	\Theta(\hat{U},U) := \text{diag}(\cos^{-1}(\sigma_{1}(U^{\top}\hat{U})), \cos^{-1}(\sigma_{2}(U^{\top}\hat{U})), \dots, \cos^{-1}(\sigma_{r}(U^{\top}\hat{U}))).
\end{equation*}
For an in-depth review of the CS decomposition and canonical angles, see for example \cite{Bhatia--1997, Stewart-Sun--1990}. An extensive summary of the relationships between $\sin\Theta$ distances, specifically $\|\sin\Theta(\hat{U},U)\|_{2}$ and $\|\sin\Theta(\hat{U},U)\|_{\rmF}$, as well as various other distance measures, is provided in \cite{Cai-Zhang--2016}. This paper focuses on $\sin\Theta$ distance in relation to Procrustes analysis.

Given two matrices $A$ and $B$ together with a set of matrices $\mathbb{S}$ and a norm $\|\cdot\|$, a general version of the Procrustes problem is given by
\begin{equation*}
	\underset{S\in\mathbb{S}}{\textnormal{inf}}\|A-BS\|.
\end{equation*}
For $U,\hat{U}\in\mathbb{O}_{p,r}$ and $\eta \in \{\rmMax, \ttinf, 2, \rmF\}$, this paper specifically considers
\begin{align}
	\label{eq:UhatUProcrustesForArbNorms}
	\underset{W\in\mathbb{O}_{r}}{\textnormal{inf}}\|\hat{U}-UW\|_{\eta}.
\end{align}
For each choice of $\eta$, the corresponding infimum in Eq.~(\ref{eq:UhatUProcrustesForArbNorms}) is provably achieved by the compactness of $\mathbb{O}_{r}$ together with properties of norms in finite-dimensional vector spaces. As such, let $W_{\eta}^{\star}\in\mathbb{O}_{r}$ denote a corresponding Procrustes solution under $\eta$ (where dependence upon the underlying matrices $U$ and $\hat{U}$ is implicit from context). Unfortunately, these solutions are not analytically tractable in general, save under the Frobenius norm, in which case $W_{U} \equiv W_{\rmF}^{\star}(U,\hat{U})$ corresponds to the the classical orthogonal Procrustes problem solution \cite{Gower-Dijksterhuis--2004} given explicitly by $W_{U} \equiv W_{1}W_{2}^{\top}$ when the singular value decomposition of $U^{\top}\hat{U} \in \R^{r \times r}$ is written as $U^{\top}\hat{U}\equiv W_{1}\Sigma_{U}W_{2}^{\top}$.

For each $\eta$, it is therefore natural to study the behavior of
\begin{align}
	\|\hat{U}-UW_{U}\|_{\eta}.
\end{align}
Towards this end, $\sin\Theta$ distances and the above Procrustes problems are related in the sense that (e.g.~\cite{Cai-Zhang--2016})
\begin{align*}
	\|\sin\Theta(\hat{U},U)\|_{\rmF}
	&\le \|\hat{U}-UW_{U}\|_{\rmF}
	\le \sqrt{2}\|\sin\Theta(\hat{U},U)\|_{\rmF},
\end{align*}
and
\begin{align*}
	\|\sin\Theta(\hat{U},U)\|_{2}
	&\le \|\hat{U}-UW_{2}^{\star}\|_{2}
	\le \|\hat{U}-UW_{U}\|_{2}
	\le \sqrt{2}\|\sin\Theta(\hat{U},U)\|_{2}.
\end{align*}
By Lemma~\ref{lem:Frob_opt_in_spectral}, $\|\hat{U}-UW_{U}\|_{2}$ can be bounded differently in a manner suggesting that the performance of $W_{U}$ is ``close'' to the performance of $W_{2}^{\star}$ under $\|\cdot\|_{2}$, namely:
\begin{align*}
	\label{eq:specW_{U}}
	\|\hat{U}-UW_{U}\|_{2}
	&\le \|\sin\Theta(\hat{U},U)\|_{2} + \|\sin\Theta(\hat{U},U)\|_{2}^{2}.	
\end{align*}
Loosely speaking, this says that the relative fluctuation between $W_{U}$ and $W_{2}^{\star}$ in the spectral norm Procrustes problem behaves as $\mathcal{O}(\|\sin\Theta(\hat{U},U)\|_{2}^{2})$.

As for the two-to-infinity norm, by simply considering the na\"{i}ve relationship between $\|\cdot\|_{2\rightarrow\infty}$ and $\|\cdot\|_{2}$, it follows that
\begin{align*}
	\tfrac{1}{\sqrt{p}}\|\sin\Theta(\hat{U},U)\|_{2}
	&\le \|\hat{U}-UW_{2\rightarrow\infty}^{\star}\|_{2\rightarrow\infty}
	\le \|\hat{U}-UW_{U}\|_{2\rightarrow\infty}.	
\end{align*}
These observations collectively suggest that direct analysis of $\hat{U}-UW_{U}$ may yield meaningfully tighter bounds on $\|\hat{U}-UW_{U}\|_{2\rightarrow\infty}$ in settings wherein $\|\hat{U}-UW_{U}\|_{2\rightarrow\infty} \ll \|\hat{U}-UW_{U}\|_{2}$ when $p \gg r$. In such a regime, $\|\cdot\|_{\ttinf}$ and $\|\cdot\|_{\rmMax}$ differ by at most a (relatively small) $r$-dependent factor, so it is conceivable that $\|\cdot\|_{\ttinf}$ may serve as a decent proxy for $\|\cdot\|_{\rmMax}$.

We now proceed to introduce a matrix perturbation framework in which $\hat{U}$ represents a perturbation (i.e.~estimate) of $U$. We then formulate a Procrustean matrix decomposition in Section~\ref{subsec:Matrix_decompositions} by further decomposing the underlying matrices whose spectral norm bounds give rise to the above quantities $\|\sin\Theta(\hat{U},U)\|_{2}$ and $\|\sin\Theta(\hat{U},U)\|_{2}^{2}$. Together with two-to-infinity norm machinery and model-based analysis, we subsequently derive a collection of operationally significant perturbation bounds and demonstrate their utility in problems of statistical estimation.

%%%%%%%%%%%%%%%%%%%%%%%%%%%%%%%%%%%%%%%%%%%%%%%%%%%
\subsection{Perturbation framework for the singular value decomposition}
\label{subsec:SVD_perturb_framework}
For rectangular matrices $X, E \in \R^{p_{1} \times p_{2}}$, let $X$ denote an unobserved matrix, let $E$ denote an unobserved perturbation (i.e.~error) matrix, and let $\hat{X}:=X+E$ denote an observed matrix that amounts to an additive perturbation of $X$ by $E$. For $X$ and $\hat{X}$, their respective partitioned singular value decompositions are given in block matrix form by
\begin{equation*}
	X = \begin{bmatrix}
	U & U_{\perp}
	\end{bmatrix}
	\cdot
	\begin{bmatrix}
	\Sigma & 0 \\
	0 & \Sigma_{\perp}
	\end{bmatrix}
	\cdot
	\begin{bmatrix}
	V^{\top} \\
	V_{\perp}^{\top}
	\end{bmatrix}
	= U \Sigma V^{\top} + U_{\perp} \Sigma_{\perp} V_{\perp}^{\top}
\end{equation*}
and
\begin{equation*}
	\hat{X} := X + E
	= 	\begin{bmatrix}
	\hat{U} & \hat{U}_{\perp}
	\end{bmatrix}
	\cdot
	\begin{bmatrix}
	\hat{\Sigma} & 0 \\
	0 & \hat{\Sigma}_{\perp}
	\end{bmatrix}
	\cdot
	\begin{bmatrix}
	\hat{V}^{\top} \\
	\hat{V}_{\perp}^{\top}
	\end{bmatrix}
	= \hat{U} \hat{\Sigma} \hat{V}^{\top} + \hat{U}_{\perp}\hat{\Sigma}_{\perp}\hat{V}_{\perp}^{\top}.
\end{equation*}
Above, $U \in \mathbb{O}_{p_{1},r}$,
$V \in \mathbb{O}_{p_{2},r}$,
$\left[U|U_{\perp}\right] \in \mathbb{O}_{p_{1}}$, and
$\left[V|V_{\perp}\right] \in \mathbb{O}_{p_{2}}$. The matrices $\Sigma \in \R^{r \times r}$ and $\Sigma_{\perp} \in \R^{(p_{1}-r) \times (p_{2}-r)}$ contain the singular values of $X$, where
$\Sigma=\text{diag}(\sigma_{1}(X),\dots,\sigma_{r}(X))$
and $\Sigma_{\perp}$ contains the remaining singular values $\sigma_{r+1}(X),\dots$ on its main diagonal, possibly padded with additional zeros, such that $\sigma_{1}(X) \ge \dots \ge \sigma_{r}(X) > \sigma_{r+1}(X) \ge \dots \ge 0$. The quantities $\hat{U},\hat{U}_{\perp},\hat{V},\hat{V}_{\perp},\hat{\Sigma}$, and $\hat{\Sigma}_{\perp}$ are defined analogously.

This paper is primarily interested in the case when $\sigma_{r}(X) \gg \sigma_{r+1}(X)$, although our results and framework hold more generally when $\Sigma$ is redefined to contain a collection of sequential singular values that are separated from the remaining singular values in $\Sigma_{\perp}$. In such a modified setting one would have $\Sigma = \textrm{diag}(\sigma_{s}(X), \dots, \sigma_{s+r}(X))$ for some positive integers $s$ and $r$, where subsequent bounds and necessary bookkeeping would depend both upon the two-sided gap $\rmMin\{\sigma_{s-1}(X)-\sigma_{s}(X),\sigma_{s+r}(X)-\sigma_{s+r+1}(X)\}$ and on the magnitude of the perturbation $E$, as in \cite{Yu-Wang-Samworth--2015}.

%%%%%%%%%%%%%%%%%%%%%%%%%%%%%%%%%%%%%%%%%%%%%%%%%%%
%%%%%%%%%%%%%%%%%%%%%%%%%%%%%%%%%%%%%%%%%%%%%%%%%%%
\section{Main results}
\label{sec:Main_results}
%%%%%%%%%%%%%%%%%%%%%%%%%%%%%%%%%%%%%%%%%%%%%%%%%%%
\subsection{A Procrustean matrix decomposition and its variants}
\label{subsec:Matrix_decompositions}
Below, Theorem~\ref{thrm:rectangular_decomp} states our main matrix decomposition in general form. Remark~\ref{rem:intuition_matrix_decomposition} subsequently provides accompanying discussion and is designed to offer a more intuitive, high-level explanation of the decomposition considerations presented here. The formal procedure for deriving Theorem~\ref{thrm:rectangular_decomp} is based on geometric considerations presented in Section~\ref{sec:rect_matrix_decomp}.
\begin{theorem}[Procrustean matrix decomposition]
	\label{thrm:rectangular_decomp}
	In the setting of Sections~\ref{subsec:SSG&Procrustes}~and~\ref{subsec:SVD_perturb_framework}, if $\hat{X}$ has rank at least $r$, then $\hat{U} - U W_{U} \in \R^{p_{1} \times r}$ admits the decomposition 
	\begin{align}
		\hat{U} - U W_{U}
		&= (I-UU^{\top}) E V W_{V} \hat{\Sigma}^{-1} \label{eq:rect_decomp_eq1}\\
		&+ (I-UU^{\top})E(\hat{V}-VW_{V})\hat{\Sigma}^{-1}\label{eq:rect_decomp_eq2}\\
		&+ (I-UU^{\top})X(\hat{V}-V V^{\top} \hat{V})\hat{\Sigma}^{-1}\label{eq:rect_decomp_eq3}\\
		&+ U(U^{\top}\hat{U}-W_{U}).\label{eq:rect_decomp_eq4}
	\end{align}
	This decomposition still holds when replacing the $r \times r$ orthogonal matrices $W_{U}$ and $W_{V}$ with any real $r \times r$ matrices $T_{1}$ and $T_{2}$, respectively. The analogous decomposition for $\hat{V}-VW_{V}$ is given by replacing $U, \hat{U}, V, \hat{V}, E, X, W_{U}$, and $W_{V}$ above with $V, \hat{V}, U, \hat{U}, E^{\top}, X^{\top}, W_{V}$, and $W_{U}$, respectively.
\end{theorem}
\begin{remark}[Intuition for Theorem~\ref{thrm:rectangular_decomp}]
	\label{rem:intuition_matrix_decomposition}
	The decomposition presented in Theorem~\ref{thrm:rectangular_decomp} can be loosely motivated in the following way. When $X$ and $\hat{X}$ have rank at least $r$, then by Section~\ref{subsec:SVD_perturb_framework}, $U \equiv XV\Sigma^{-1}$ and $\hat{U} \equiv \hat{X}\hat{V}\hat{\Sigma}^{-1} = X\hat{V}\hat{\Sigma}^{-1} + E\hat{V}\hat{\Sigma}^{-1}$. It is thus conceivable that the difference between $U$ and $\hat{U}$ behaves to leading order as $EV\Sigma^{-1}$ (modulo proper orthogonal transformation) under suitable perturbation and structural assumptions. Indeed, we repeatedly observe such first-order behavior via the matrix term $(I-UU^{\top}) E V W_{V} \hat{\Sigma}^{-1}$ when $\|U\|_{\ttinf} \ll 1$.
	
	For the purpose of obtaining upper bounds, passing between $\Sigma^{-1}$ and $\hat{\Sigma}^{-1}$ amounts to transitioning between $\sigma_{r}(X)$ and $\sigma_{r}(\hat{X})$; this can be done successfully via Weyl's inequality \cite{Bhatia--1997} provided the perturbation $E$ is suitably small in norm relative to $\sigma_{r}(X)$.
	
	Subsequent results in Section~\ref{subsec:Perturbation_theorems} will demonstrate that Lines~(\ref{eq:rect_decomp_eq2})--(\ref{eq:rect_decomp_eq4}) amount to circumstance-driven residual and approximation error terms. Namely, with respect to the two-to-infinity norm:
	\begin{itemize}
		\item Line~(\ref{eq:rect_decomp_eq2}) can be much smaller than $\|\sin\Theta(\hat{V},V)\|_{2}$ as a function of the relative magnitudes of $E$ and $\hat{\Sigma}^{-1}$.
		\item Line~(\ref{eq:rect_decomp_eq3}) can be much smaller than $\|\sin\Theta(\hat{V},V)\|_{2}$ as a function of the multiplicative singular value gap $\sigma_{r+1}(X)/\sigma_{r}(\hat{X})$.
		\item Line~(\ref{eq:rect_decomp_eq4}) can be much smaller than $\|\sin\Theta(\hat{U},U)\|_{2}^{2}$ as a function of $\|U\|_{\ttinf}$, specifically when $\|U\|_{\ttinf} \ll 1$.
	\end{itemize}
\end{remark}
Theorem~\ref{thrm:rectangular_decomp} can be rewritten in terms of the spectral matrix decomposition when $X$ and $E$ are both symmetric matrices. For ease of reference, we state this special case in the form of a corollary.
\begin{corollary}
	\label{cor:symmetric_decomp}
	Let $X, E \in \R^{p \times p}$ be symmetric matrices. Rephrase Section~\ref{subsec:SVD_perturb_framework} to hold for the spectral matrix decomposition in terms of the eigenvalues and eigenvectors of $X$ and $\hat{X}$. Provided $\hat{X}$ has rank at least $r$, then $\hat{U} - U W_{U} \in \R^{p \times r}$ admits the decomposition
	\begin{align}
		\hat{U} - U W_{U}
		&= (I-UU^{\top}) E U W_{U} \hat{\Sigma}^{-1}\\
		&+ (I-UU^{\top})E(\hat{U}-UW_{U})\hat{\Sigma}^{-1}\nonumber\\
		&+ (I-UU^{\top})X(\hat{U}-UU^{\top}\hat{U})\hat{\Sigma}^{-1}\nonumber\\
		&+ U(U^{\top}\hat{U}-W_{U}).\nonumber
	\end{align}
\end{corollary}

\begin{remark}[The orthogonal matrix $W_{U}$]
	\label{rem:orthogonal_matrix_WU}
	To reiterate, $W_{U}$ depends upon the perturbed quantity $\hat{U}$ which in turn depends upon the perturbation $E$. Consequently, $W_{U}$ is unknown (resp.~random) when $E$ is assumed unknown (resp.~random). This paper makes no distinct singular value (resp.~eigenvalue) assumption, so in general the quantity $\hat{U}$ alone cannot hope to recover $U$ in the presence of singular value multiplicity. Indeed, $\hat{U}$ can only be viewed as an estimate of $U$ up to an orthogonal transformation, and our specific choice of $W_{U}$ is based upon the aforementioned Procrustes-based considerations. We note that statistical inference methodologies and applications are often either invariant under or equivalent modulo orthogonal transformations as a source of non-identifiability. For example, $K$-means clustering applied to the rows of $U$ in Euclidean space is equivalent to clustering the rows of the matrix $UW_{U}$.
\end{remark}

It will subsequently prove convenient to work with the following modified versions of Theorem~\ref{thrm:rectangular_decomp} which are stated below as corollaries.

\begin{corollary}
	\label{cor:rectangular_rewritten}
	The decomposition in Theorem~\ref{thrm:rectangular_decomp} can be rewritten as
	\begin{align}
		\hat{U}-UW_{U} &=  (I-UU^{\top})E(VV^{\top})VW_{V}\hat{\Sigma}^{-1}\\
		&+ (I-UU^{\top})(E+X)(\hat{V}-VW_{V})\hat{\Sigma}^{-1}\nonumber\\
		&+ U(U^{\top}\hat{U}-W_{U}).\nonumber
	\end{align}
\end{corollary}

\begin{corollary}
	\label{cor:rectangular_expanded}
	Corollary~\ref{cor:rectangular_rewritten} can be equivalently written as
	\begin{align}
		\hat{U}-UW_{U}
		&= (U_{\perp}U_{\perp}^{\top})E(VV^{\top})VW_{V}\hat{\Sigma}^{-1}\\ 
		&+ (U_{\perp}U_{\perp}^{\top})E(VV^{\top})V(V^{\top}\hat{V}-W_{V})\hat{\Sigma}^{-1}\nonumber\\
		&+ (U_{\perp}U_{\perp}^{\top})E(V_{\perp}V_{\perp}^{\top})(\hat{V}-VV^{\top}\hat{V})\hat{\Sigma}^{-1}\nonumber\\
		&+ (U_{\perp}U_{\perp}^{\top})X(V_{\perp}V_{\perp}^{\top})(\hat{V}-VV^{\top}\hat{V})\hat{\Sigma}^{-1}\nonumber\\
		&+ U(U^{\top}\hat{U}-W_{U}).\nonumber
	\end{align}
\end{corollary}

For both Corollary~\ref{cor:rectangular_rewritten}~and~\ref{cor:rectangular_expanded}, the first term following the equality sign in each display equation is shown to be the leading order term of interest in practice. This point shall be made more precise and quantitative below.

%%%%%%%%%%%%%%%%%%%%%%%%%%%%%%%%%%%%%%%%%%%%%%%%%%%
%%%%%%%%%%%%%%%%%%%%%%%%%%%%%%%%%%%%%%%%%%%%%%%%%%%
\subsection{General perturbation theorems}
\label{subsec:Perturbation_theorems}
This section presents a collection of perturbation theorems derived via a unified methodology that combines Theorem~\ref{thrm:rectangular_decomp}, its variants, the two-to-infinity norm machinery in Section~\ref{sec:Tech_tools_2infty}, and the geometric observations in Section~\ref{sec:Geometric_bounds}. We bound $\hat{U}-UW_{U}$, while similar bounds for $\hat{V}-VW_{V}$ hold under the appropriate modifications detailed in Theorem~\ref{thrm:rectangular_decomp}.

%%%%%%%%%%%%%%%%%%%%%%%
\begin{theorem}[Baseline two-to-infinity norm bound]
	\label{thrm:baselineProcrustesPerturbBound}
	Provided $\sigma_{r}(X)>\sigma_{r+1}(X) \ge 0$ and $\sigma_{r}(X) \ge 2\|E\|_{2}$, then
	\begin{align}
	\|\hat{U}-UW_{U}\|_{\ttinf}
	&\le
	2\left(\frac{\|(U_{\perp}U_{\perp}^{\top})E(VV^{\top})\|_{\ttinf}}{\sigma_{r}(X)}\right)\\
	&+ 2\left(\frac{\|(U_{\perp}U_{\perp}^{\top})E(V_{\perp}V_{\perp}^{\top})\|_{\ttinf}}{\sigma_{r}(X)}\right)\|\sin\Theta(\hat{V},V)\|_{2}\nonumber\\
	&+ 2\left(\frac{\|(U_{\perp}U_{\perp}^{\top})X(V_{\perp}V_{\perp}^{\top})\|_{\ttinf}}{\sigma_{r}(X)}\right)\|\sin\Theta(\hat{V},V)\|_{2}\nonumber\\
	&+ \|\sin\Theta(\hat{U},U)\|_{2}^{2}\|U\|_{\ttinf}.\nonumber
	\end{align}	
\end{theorem}
%%%%%%%%%%%%%%%%%%%%%%%
Let $C_{X,U}$ and $C_{X,V}$ denote upper bounds on the quantities $\|(U_{\perp}U_{\perp}^{\top})X\|_{\infty}$ and $\|(V_{\perp}V_{\perp}^{\top})X^{ \top}\|_{\infty}$, respectively, and define $C_{E,U},C_{E,V}$ analogously. Theorem~\ref{thrm:General_rectangular_bounds} provides a uniform perturbation bound for $\|\hat{U}-UW_{U}\|_{2\rightarrow\infty}$ and $\|\hat{V}-VW_{V}\|_{2\rightarrow\infty}$. When $\rmRank(X)=r$, Corollary~\ref{cor:reFanThrmRect} presents a weaker but simpler version of the bound in Theorem~\ref{thrm:General_rectangular_bounds}.
\begin{theorem}[Uniform perturbation bound for rectangular matrices]
	\label{thrm:General_rectangular_bounds}
	Suppose $\sigma_{r}(X)>\sigma_{r+1}(X)>0$ and that
	\begin{equation*}
		\sigma_{r}(X) \ge \textnormal{max}\{2\|E\|_{2},(2/\alpha)C_{E,U},(2/\alpha^{\prime})C_{E,V},(2/\beta)C_{X,U},(2/\beta^{\prime})C_{X,V}\}
	\end{equation*}
	for constants $0<\alpha,\alpha^{\prime},\beta,\beta^{\prime}<1$ such that $\delta:=(\alpha+\beta)(\alpha^{\prime}+\beta^{\prime})<1$. Then,
	\begin{align}
		(1-\delta)\|\hat{U}-UW_{U}\|_{2\rightarrow\infty}
		&\le 2\left(\frac{\|(U_{\perp}U_{\perp}^{\top})E(VV^{\top})\|_{\ttinf}}{\sigma_{r}(X)}\right)\\
		&+ 2\left(\frac{\|(V_{\perp}V_{\perp}^{\top})E^{\top}(UU^{\top})\|_{\ttinf}}{\sigma_{r}(X)}\right)\nonumber\\
		&+ \|\sin\Theta(\hat{U},U)\|_{2}^{2}\|U\|_{\ttinf}\nonumber\\
		&+ \|\sin\Theta(\hat{V},V)\|_{2}^{2}\|V\|_{\ttinf}.\nonumber
	\end{align}
	If $\textnormal{rank}(X)=r$ so that $\sigma_{r+1}(X)=0$, then the above bound holds for $\delta:=\alpha\times\alpha^{\prime}<1$ under the weaker assumption that 
	\begin{equation*}
		\sigma_{r}(X) \ge \textnormal{max}\{2\|E\|_{2},(2/\alpha)C_{E,U},(2/\alpha^{\prime})C_{E,V}\}.
	\end{equation*}
\end{theorem}

\begin{corollary}[Uniform perturbation bound for low rank matrices]
	\label{cor:reFanThrmRect}
	Suppose $\sigma_{r}(X)>\sigma_{r+1}(X)=0$ and that
	\begin{equation*}
		\sigma_{r}(X) \ge \textnormal{max}\{2\|E\|_{2},(2/\alpha)C_{E,U},(2/\alpha^{\prime})C_{E,V}\}
	\end{equation*}
	for some constants $0<\alpha,\alpha^{\prime}<1$ such that $\delta:=\alpha\times\alpha^{\prime}<1$. Then
	\begin{equation}
		(1-\delta)\|\hat{U}-UW_{U}\|_{\ttinf}
		\le 12
		\times
		\underset{\eta \in \{1, \infty\}}{\textnormal{max}}\left\{\frac{\|E\|_{\eta}}{\sigma_{r}(X)}\right\}\times
		\underset{Z \in \{U, V\}}{\textnormal{max}}\left\{\|Z\|_{\ttinf}\right\}.
	\end{equation}
\end{corollary}

%%%%%%%%%%%%%%%%%%%%%%%%%%%%%%%%%%%%%%%%%%%%%%%%%%%
%%%%%%%%%%%%%%%%%%%%%%%%%%%%%%%%%%%%%%%%%%%%%%%%%%%
\section{Applications}
\label{sec:Applications}
This section applies our perturbation theorems and two-to-infinity norm machinery to three statistical settings corresponding to, among others, the results in \cite{Fan-et-al--2016}, \cite{Cai-Zhang--2016}, and \cite{Lyzinski-et-al--2014EJS}, thereby yielding Theorem~\ref{thrm:FanThrm2.1++}, Theorem~\ref{thrm:SubspaceRecovery}, and Theorem~\ref{thrm:rhoSBM2infinity}, respectively. Each of these theorems (including Theorem~\ref{thrm:generalCovarEst} presented earlier) is obtained by combining general considerations with application-specific analysis.

Moving forward, consider the following structural matrix property which arises within the context of low rank matrix recovery.
\begin{definition}[\cite{Candes-Recht--2009}]
	Let $\mathcal{U}$ be a subspace of $\R^{p}$ of dimension $r$, and let $P_{\mathcal{U}}$ be the orthogonal projection onto $\mathcal{U}$. Then the \emph{coherence} of $\mathcal{U}$ (vis-\`{a}-vis the standard basis $\{e_{i}\}$) is defined to be
	\begin{equation}
		\mu(\mathcal{U}) := \left(\frac{p}{r}\right)\underset{i\in[p]}{\text{max}}\|P_{\mathcal{U}}e_{i}\|_{2}^{2}.
	\end{equation}
\end{definition}
For $U\in\mathbb{O}_{p,r}$, the (orthonormal) columns of $U$ span a subspace of dimension $r$ in $\R^{p}$, so it is natural to abuse notation and to interchange $U$ with its underlying subspace $\mathcal{U}$. In this case $P_{U} \equiv UU^{\top}$, and so Propositions~\ref{prop:2infty_as_max_Euclidean}~and~\ref{prop:partial_isometry_invariance} yield the equivalent formulation
\begin{equation*}
	\mu(U) :=\left(\frac{p}{r}\right)\|U\|_{\ttinf}^{2}.
\end{equation*}
Observe that $1 \le \mu(U) \le p/r$, where the upper and lower bounds are achieved for $U$ consisting of all standard basis vectors and of vectors whose entries each have magnitude $1/\sqrt{p}$, respectively. Since the columns of $U$ are mutually orthogonal with unit Euclidean norm, the magnitude of $\mu(U)$ can be viewed as describing the row-wise accumulation of ``mass'' in $U$.

The bounded coherence property \cite{Candes-Recht--2009} corresponds to the existence of a positive constant $C_{\mu} \ge 1$ such that
\begin{equation}
	\label{eq:boundedCoherencePropertyTTINF}
	\|U\|_{2\rightarrow\infty}\le C_{\mu}\sqrt{\frac{r}{p}}.
\end{equation}
Bounded coherence arises naturally in the random orthogonal (matrix) model and influences the recoverability of low rank matrices via nuclear norm minimization when sampling only a subset of the matrix entries \cite{Candes-Recht--2009}. In random matrix theory, bounded coherence is closely related to eigenvector delocalization \cite{Rudelson-Vershynin--2015}. Examples of matrices whose row and column space factors exhibit bounded coherence can be found, for example, in the study of networks. Specifically, it is not difficult to check that bounded coherence holds for the top eigenvectors of the (non-random) low rank edge probability matrices corresponding to the Erd\H{o}s-R\'{e}nyi random graph model and the balanced $K$-block stochastic block model, among others.
%%%%%%%%%%%%%%%%%%%%%%%%%%%%%%%%%%%%%%%%%%%%%%%%%%%
%%%%%%%%%%%%%%%%%%%%%%%%%%%%%%%%%%%%%%%%%%%%%%%%%%%
\subsection{Singular vector (entrywise) perturbation bound}
\label{subsec:Application_Fan}
In \cite{Fan-et-al--2016} the authors consider low rank matrices exhibiting bounded coherence. For such matrices, the results therein provide singular vector perturbation bounds under the $\ell_{\infty}$ vector norm which are then applied to robust covariance estimation.

In this paper, Corollary~\ref{cor:reFanThrmRect} formulates a perturbation bound that is operationally in the same spirit as results in \cite{Fan-et-al--2016}. Note that upper bounding $\|\hat{U}-UW_{U}\|_{\ttinf}$ immediately bounds both $\|\hat{U}-UW_{U}\|_{\rmMax}$ and $\textnormal{inf}_{W\in\mathbb{O}_{r}}\|\hat{U}-UW\|_{\rmMax}$, thereby providing $\ell_{\infty}$-type bounds for the perturbed singular vectors up to orthogonal transformation, the analogue of sign flips for well-separated, distinct singular values (similarly for $V$, $\hat{V}$, and $W_{V}$). The joint, symmetric nature of the singular value gap assumption controls the dependence of $\|\hat{U}-UW_{U}\|_{2\rightarrow\infty}$ and $\|\hat{V}-VW_{V}\|_{2\rightarrow\infty}$ on one another and takes into account the underlying matrix dimensions.

For symmetric matrices, Theorem~\ref{thrm:FanThrm2.1++} improves upon \cite{Fan-et-al--2016} and implicitly applies to the applications discussed therein.

\begin{theorem}[Application:~eigenvector (entrywise) perturbation bound]
	\label{thrm:FanThrm2.1++}
	Let $X,E \in \R^{p \times p}$ be symmetric matrices where $X$ with $\textnormal{rank}(X)=r$ has spectral decomposition $X=U \Lambda U^{\top} + U_{\perp}\Lambda_{\perp}U_{\perp}^{\top} \equiv U\Lambda U^{\top}$ and leading eigenvalues $|\lambda_{1}|\ge|\lambda_{2}|\ge \dots \ge|\lambda_{r}| > 0$. Suppose that $|\lambda_{r}| \ge 4\|E\|_{\infty}$. Then there exists an orthogonal matrix $W\in\mathbb{O}_{r}$ such that
	\begin{align}
		\|\hat{U}-UW\|_{2\rightarrow\infty}
		\le 14\left(\frac{\|E\|_{\infty}}{|\lambda_{r}|}\right)\|U\|_{\ttinf}.
	\end{align}
\end{theorem}
Theorem~\ref{thrm:FanThrm2.1++} provides a user-friendly, deterministic perturbation bound that permits repeated eigenvalues. Theorem~\ref{thrm:FanThrm2.1++} makes no assumption on $\|U\|_{\ttinf}$, though it can be immediately combined with the bounded coherence property reflected in Eq.~(\ref{eq:boundedCoherencePropertyTTINF}) to yield the bound
\begin{align*}
	\|\hat{U}-UW\|_{\ttinf}
	\le 14C_{\mu}\left(\frac{\sqrt{r}\|E\|_{\infty}}{\sqrt{p}|\lambda_{r}|}\right).
\end{align*}
It is worth emphasizing that stronger (albeit more complicated) bounds are obtained in the proof leading up to the statement of Theorem~\ref{thrm:FanThrm2.1++}.

%%%%%%%%%%%%%%%%%%%%%%%%%%%%%%%%%%%%%%%%%%%%%%%%%%%
%%%%%%%%%%%%%%%%%%%%%%%%%%%%%%%%%%%%%%%%%%%%%%%%%%%
\subsection{Singular subspace perturbation and random matrices}
\label{subsec:Application_Cai}
This section interfaces the results in this paper with the spectral and Frobenius-based rate-optimal singular subspace perturbation bounds obtained in \cite{Cai-Zhang--2016}.

Consider the setting wherein $X\in\R^{p_{1} \times p_{2}}$ is a fixed rank $r$ matrix with $r \ll p_{1} \ll p_{2}$ and $\sigma_{r}(X) \ge C(p_{2}/\sqrt{p_{1}})$, where $E\in\R^{p_{1} \times p_{2}}$ is a random matrix with independent standard normal entries. By \cite{Cai-Zhang--2016}, then with high probability, the following bounds hold for the left and right singular subspaces:
\begin{equation*}
	\|\sin\Theta(\hat{U},U)\|_{2}
	\le C\left(\frac{\sqrt{p_{1}}}{\sigma_{r}(X)}\right)
	\text{  and  }
	\|\sin\Theta(\hat{V},V)\|_{2}
	\le C\left(\frac{\sqrt{p_{2}}}{\sigma_{r}(X)}\right).
\end{equation*}
Here, working with $V$ and $\hat{V}$ is desirable though comparatively more difficult. Theorem~\ref{thrm:SubspaceRecovery} demonstrates how (even relatively coarse) two-to-infinity norm analysis allows one to recover upper and lower bounds for $\|\hat{V}-VW_{V}\|_{\ttinf}$ that at times nearly match. For ease of presentation, Theorem~\ref{thrm:SubspaceRecovery} is stated simply as holding with high probability.
\begin{theorem}[Application:~singular subspace recovery]
	\label{thrm:SubspaceRecovery}
	Let $X,E\in\R^{p_{1} \times p_{2}}$ be as in Section~\ref{subsec:Application_Cai}. Then there exists a constant $C_{r}>0$ such that with high probability,
	\begin{equation}
		\|\hat{V}-VW_{V}\|_{\ttinf}
		\le C_{r}\left(\tfrac{\log(p_{2})}{\sigma_{r}(X)}\right)
		\left(1
		+ \left(\tfrac{p_{1}}{\sigma_{r}(X)}\right)
		+ \left(\tfrac{\sqrt{p_{1}}}{\log(p_{2})}\right)\|V\|_{\ttinf}
		\right).
	\end{equation}
	If in addition $\sigma_{r}(X) \ge c p_{1}$ and $\|V\|_{\ttinf}\le c_{r}/\sqrt{p_{2}}$ for some $c, c_{r} > 0$, then with high probability
	\begin{equation}
		\|\hat{V}-VW_{V}\|_{\ttinf}
		\le C_{r}\left(\tfrac{\log(p_{2})}{\sigma_{r}(X)}\right).
	\end{equation}
	The lower bound $\tfrac{1}{\sqrt{p_{2}}}\|\sin\Theta(\hat{V},V)\|_{2}
	\le \|\hat{V}-VW_{V}\|_{\ttinf}$ always holds by Proposition~\ref{prop:2infty_&_spec} and Lemma~\ref{lemma:geom_resid_bounds}.
\end{theorem}
%%%%%%%%%%%%%%%%%%%%%%%%%%%%%%%%%%%%%%%%%%%%%%%%%%%
%%%%%%%%%%%%%%%%%%%%%%%%%%%%%%%%%%%%%%%%%%%%%%%%%%%
\subsection{Statistical inference for random graphs}
\label{subsec:Application_Graphs}
In the study of networks, community detection and clustering are tasks of central interest. A network (i.e.~a graph $\mathcal{G}\equiv(\mathcal{V},\mathcal{E})$ consisting of a vertex set $\mathcal{V}$ and edge set $\mathcal{E}$) may be represented by its adjacency matrix, $A \equiv A_{\mathcal{G}}$, which captures the edge connectivity of the nodes in the network. For inhomogeneous independent edge random graph models, the adjacency matrix can be viewed as a random perturbation of an underlying (often low rank) edge probability matrix $P$, where in expectation $P \equiv \mathbb{E}[A]$. In the notation of Section~\ref{subsec:SVD_perturb_framework}, the matrix $P$ corresponds to $X$, the matrix $A-P$ corresponds to $E$, and the matrix $A$ corresponds to $\hat{X}$. By viewing $\hat{U}$ (here the matrix of leading eigenvectors of $A$) as an estimate of $U$ (here the matrix of leading eigenvectors of $P$), the results in Section~\ref{sec:Main_results} immediately apply.

Spectral methods and related optimization problems for random graphs employ the spectral decomposition of the adjacency matrix (or matrix-valued functions thereof, e.g.~the Laplacian matrix and its variants) \cite{Tang-Priebe--2016+, Sussman-et-al--2012, Rohe-Chatterjee-Yu--2011, sarkar2015role}. For example, \cite{Le-Vershynin-Levina--2016} presents a general spectral-based, dimension-reduction community detection framework which incorporates the spectral norm distance between the leading eigenvectors of $A$ and $P$. Taken in the context of \cite{Le-Vershynin-Levina--2016} and indeed the wider statistical network analysis literature, this paper complements existing work and paves the way for expanding the toolkit of network analysts to include more Procrustean considerations and two-to-infinity norm machinery.

Much of the existing literature for graphs and network models concerns the popular stochastic block model (SBM) \cite{Holland-et-al--1983} and its variants. The related random dot product graph model (RDPG model) \cite{Young-Scheinerman--2007} has recently been developed in a series of papers as both a tractable and flexible random graph model amenable to spectral methods \cite{Sussman-et-al--2012, Fishkind-et-al--2013, Sussman-et-al--2014, Lyzinski-et-al--2014EJS, Tang-et-al--2014, Tang-et-al--2016, Tang-Priebe--2016+}. In the RDPG model, the graph's eigenvalues and eigenvectors are closely related to the model's generative latent positions. In particular, the leading eigenvectors of the adjacency matrix can be used to estimate the latent positions when properly scaled by the leading eigenvalues.

In the context of the wider RDPG literature, this paper extends both the treatment of the two-to-infinity norm in \cite{Lyzinski-et-al--2014EJS} and Procrustes matching for graphs in \cite{Tang-et-al--2016}. Our two-to-infinity norm bounds in Section~\ref{sec:Main_results} imply an eigenvector version of Lemma 5 in \cite{Lyzinski-et-al--2014EJS} that does not require the matrix-valued model parameter $P$ to have distinct eigenvalues. Our Procrustes analysis also suggests a refinement of the test statistic formulation in the two-sample graph inference hypothesis testing framework of \cite{Tang-et-al--2016}.

It is worth mentioning that our level of generality permits the consideration of random graph (matrix) models that allow edge dependence structure, such as the $(C,c,\gamma)$ property \cite{O'Rourke-Vu-Wang--2014} (see below). Indeed, moving beyond independent edge models represents an important direction for future work in the field of statistical network analysis.

\begin{definition}[\cite{O'Rourke-Vu-Wang--2014}]
	\label{def:Ccgamma}
	A $p_{1} \times p_{2}$ random matrix $E$ is said to be $(C,c,\gamma)$-concentrated if, given a trio of positive constants $(C,c,\gamma)$, for all unit vectors $u\in\R^{p_{1}}$, $v \in \R^{p_{2}}$, and for every $t>0$,
	\begin{equation}
	\mathbb{P}\left[|\langle E v, u \rangle| > t \right]
	\le C \exp(-c t^{\gamma}).
	\end{equation}
\end{definition}

\begin{remark}[Probabilistic concentration and the perturbation $E$]
	\label{rem:error_matrix_concentration}
	The proofs of our main results demonstrate the importance of bounding $\|EV\|_{2\rightarrow\infty}$ and $\|U^{\top}EV\|_{2}$ in the perturbation framework of Section \ref{subsec:SVD_perturb_framework}. When $E$ satisfies the $(C,c,\gamma)$-concentrated property in Definition \ref{def:Ccgamma}, these quantities can be easily controlled by simple union bounds. In general, it is desirable to control these quantities via standard Bernstein and Hoeffding-type probabilistic bounds encountered throughout statistics.
\end{remark}

In the statistical network analysis literature, current active research directions include the development of random graph models exhibiting edge correlation and the development of inference methodology for multiple graphs. Here, we briefly consider the $\rho$-correlated stochastic block model \cite{lyzinski2015spectral} and the omnibus embedding matrix for multiple graphs \cite{Priebe--2013} employed in \cite{Chen--2016, Lyzinski--JOFC--2016, levin2017central}. The $\rho$-correlated stochastic block model provides a simple yet easily interpretable and tractable model for dependent random graphs \cite{Lyzinski--info--2017+}, while the omnibus embedding matrix provides a framework for performing spectral analysis on multiple graphs by leveraging graph (dis)similarities \cite{Priebe--2013, Lyzinski--JOFC--2016,Chen--2016}.

\begin{definition}[\cite{Lyzinski--info--2017+}]
	\label{def:rhoSBMgraphs}
	Let $\mathcal{G}^{n}$ denote the set of labeled, $n$-vertex, simple, undirected graphs.
	Two $n$-vertex random graphs $(G^{1},G^{2}) \in \mathcal{G}^{1}\times\mathcal{G}^{2}$ are said to be \emph{$\rho$-correlated SBM$(\kappa,\overset{\rightarrow}{n},b,\Lambda)$ graphs} (abbreviated $\rho$-SBM) if
	\begin{enumerate}
		\item $G^{1}:=(\mathcal{V},\mathcal{E}(G^{1}))$ and $G^{2}:=(\mathcal{V},\mathcal{E}(G^{2}))$ are marginally SBM$(\kappa,\overset{\rightarrow}{n},b,\Lambda)$ random graphs; i.e.\ for each $i=1,2$,
		\begin{enumerate}
			\item The vertex set $\mathcal{V}$ is the union of $\kappa$ blocks $\mathcal{V}_{1},\mathcal{V}_{2},\dots,\mathcal{V}_{\kappa}$, which are disjoint sets with respective cardinalities $n_{1},n_{2},\dots,n_{\kappa}$;
			\item The block membership function $b:\mathcal{V}\mapsto[\kappa]$ is such that for each $v\in\mathcal{V}$, $b(v)$ denotes the block of $v$; i.e., $v \in \mathcal{V}_{b(v)}$;
			\item The block adjacency probabilities are given by the symmetric matrix $\Lambda\in[0,1]^{\kappa\times \kappa}$; i.e., for each pair of vertices $\{j,l\}\in\binom{\mathcal{V}}{2}$, the adjacency of $j$ and $l$ is an independent Bernoulli trial with probability of success $\Lambda_{b(j),b(l)}$.
		\end{enumerate}
		\item The random variables
		\begin{equation*}
		\{\mathbb{I}[\{j,k\}\in\mathcal{E}(G^{i})]\}_{i=1,2;\{j,k\}\in\binom{\mathcal{V}}{2}}
		\end{equation*}
		are collectively independent except that for each $\{j,k\}\in\binom{\mathcal{V}}{2}$, the correlation between $\mathbb{I}[\{j,k\}\in\mathcal{E}(G^{1})]$ and $\mathbb{I}[\{j,k\}\in\mathcal{E}(G^{2})]$ is $\rho \ge 0$.
	\end{enumerate}
\end{definition}

The following theorem provides a guarantee for estimating the leading eigenvectors of a multiple graph omnibus matrix when the graphs are not independent. Theorem~\ref{thrm:rhoSBM2infinity} is among the first of its kind and complements the recent, concurrent work on joint graph embedding in \cite{levin2017central}.

\begin{theorem}[Application:~multiple graph inference]
	\label{thrm:rhoSBM2infinity}
	Let $(G^{1},G^{2})$ be a pair of $\rho$-correlated SBM$(\kappa,\overset{\rightarrow}{n},b,\Lambda)$ graphs as in Definition~\ref{def:rhoSBMgraphs} with $n \times n$ (symmetric, binary) adjacency matrices $(A^{1},A^{2})$. Let the model omnibus matrix $\mathfrak{O}$ and adjacency omnibus matrix $\hat{\mathfrak{O}}$ be given by
	\begin{align*}
		\mathfrak{O} :=
		\begin{bmatrix}
		1 & 1 \\
		1 & 1
		\end{bmatrix}
		\otimes \mathcal{Z}\Lambda\mathcal{Z}^{\top}
		\textnormal{ and }
		\hat{\mathfrak{O}} :=
		\begin{bmatrix}
		A^{1} & \frac{A^{1}+A^{2}}{2} \\
		\frac{A^{1}+A^{2}}{2} & A^{2}
		\end{bmatrix},
	\end{align*}
	where $\otimes$ denotes the matrix Kronecker product and  $\mathcal{Z}$ is the $n \times \kappa$ matrix of vertex-to-block assignments such that $P:=\mathcal{Z}\Lambda\mathcal{Z}^{\top}\in[0,1]^{n \times n}$ denotes the edge probability matrix. Let $\textnormal{rank}(\Lambda)=r$ and hence $\textnormal{rank}(\mathfrak{O})=r$. For $i=1,2$, suppose that the maximum expected degree of $G^{i}$, $\Delta$, satisfies $\Delta \gg \log^{4}(n)$, along with $\sigma_{r}(\mathfrak{O}) \ge c\Delta$ for some $c>0$. Let $U,\hat{U}\in\mathbb{O}_{2n,r}$ denote the matrices whose columns are the normalized eigenvectors corresponding to the largest eigenvalues of $\mathfrak{O}$ and $\hat{\mathfrak{O}}$ given by the diagonal matrices $\Sigma$ and $\hat{\Sigma}$, respectively. Then with probability $1-o(1)$ as $n \rightarrow \infty$,
	\begin{align*}
		\|\hat{U} - U W_{U}\|_{\ttinf}
		=\mathcal{O}_{r}\left(\tfrac{\log n}{\Delta}\right).
	\end{align*}
	In contrast, spectral norm analysis implies the weaker two-to-infinity norm bound
	$\|\hat{U}-UW_{U}\|_{\ttinf} = \mathcal{O}_{r}\left(\tfrac{1}{\sqrt{\Delta}}\right)$.
\end{theorem}
\begin{remark}[Edge correlation]
	\label{rem:edge_correlation}
	The implicit dependence upon the correlation factor $\rho$ in Theorem~\ref{thrm:rhoSBM2infinity} can be made explicit by a more careful analysis of constant factors and the probability statement. This is not our present concern.
\end{remark}

%%%%%%%%%%%%%%%%%%%%%%%%%%%%%%%%%%%%%%%%%%%%%%%%%%%
%%%%%%%%%%%%%%%%%%%%%%%%%%%%%%%%%%%%%%%%%%%%%%%%%%%
\section{Discussion and Conclusion}
\label{sec:Discussion}
This paper develops a flexible Procrustean matrix decomposition and its variants together with machinery for the two-to-infinity norm in order to study the perturbation of singular vectors and subspaces. We have demonstrated both implicitly and explicitly the widespread applicability of our framework and results to a host of popular matrix noise models, namely matrices that have
\begin{itemize}
	\item independent and identically distributed entries (Section~\ref{subsec:Application_Cai});
	\item independent and identically distributed rows (Section~\ref{subsec:Sample_application});
	\item independent but not identically distributed entries (Section~\ref{subsec:Application_Graphs});
	\item neither independent nor identically distributed entries (Section~\ref{subsec:Application_Fan}).
\end{itemize}

Each application presented in this paper requires model-specific analysis. Namely, one must determine which formulation of the Procrustean matrix decomposition to use, how to effectively transition between norms, and how to analyze the resulting quantities. For example, in Section~\ref{subsec:Application_Fan} the product term $\|E\|_{\infty}\|U\|_{\ttinf}$ is meaningful when coupled with the bounded coherence assumption, whereas the term $\|EU\|_{\ttinf}$ is analyzed directly in order to prove Theorem~\ref{thrm:rhoSBM2infinity}. Similarly, with respect to covariance estimation (Theorems~\ref{thrm:generalCovarEst}~and~\ref{thrm:FanThrm2.1++}), context-specific differences motivate idiosyncratic approaches when deriving the stated bounds.

This paper focuses on decomposing the matrix $\hat{U}-UW_{U}$ and on establishing the two-to-infinity norm as a useful tool for matrix perturbation analysis. In the time since this work was first made publicly available, there has been a flurry of activity within the statistics, computer science, and mathematics communities devoted to obtaining refined entrywise singular vector and eigenvector perturbation bounds \cite{cape2018signal,mao2017estimating,abbe2017entrywise,eldridge2017unperturbed,tang2017asymptotically}. Among the observations made earlier in this paper, it is useful to keep in mind that
\begin{equation*}
	\underset{W\in\mathbb{O}_{r}}{\textnormal{inf}}\|\hat{U}-UW\|_{\rmMax}
	\le \|\hat{U} - UW_{U}\|_{\rmMax}
	\le \|\hat{U} - U W_{U}\|_{\ttinf}.
\end{equation*}

Ample open problems and applications exist for which it is and will be productive to utilize the two-to-infinity norm and matrix decompositions in the future. It is our hope that the level of generality and flexibility presented in this paper will facilitate the more widespread use of the two-to-infinity norm in statistics.

%%%%%%%%%%%%%%%%%%%%%%%%%%%%%%%%%%%%%%%%%%%%%%%%%%%
%%%%%%%%%%%%%%%%%%%%%%%%%%%%%%%%%%%%%%%%%%%%%%%%%%%

\section{Proofs}
\label{sec:Proofs}
%%%%%%%%%%%%%%%%%%%%%%%%%%%%%%%%%%%%%%%%%%%%%%%%%%%
%%%%%%%%%%%%%%%%%%%%%%%%%%%%%%%%%%%%%%%%%%%%%%%%%%%
	\subsection{Technical tools for the two-to-infinity norm}
	\label{sec:Tech_tools_2infty}
	For $A \in \R^{p_{1} \times p_{2}}$, consider the vector norm on matrices $\|\cdot\|_{\ttinf}$ defined by
	\begin{equation*}
		\|A\|_{\ttinf}
		:= \underset{\|x\|_2=1}{\text{sup}}\|Ax\|_{\infty}.
	\end{equation*}
	Let $e_{i}$ denote the $i$th standard basis vector, and let $A_{i} \in \R^{p_{2}}$ denote the $i$th row of $A$. The following proposition says that $\|A\|_{\ttinf}$ corresponds to the maximum row-wise Euclidean norm of $A$.
	\begin{proposition}
		\label{prop:2infty_as_max_Euclidean}
		For $A \in \R^{p_{1} \times p_{2}}$, then $\|A\|_{\ttinf} = \underset{i \in [p_{1}]}{\textnormal{max}}\|A_{i}\|_{2}$.
	\end{proposition}
	\begin{proof}
		The definition of $\|\cdot\|_{\ttinf}$ and the Cauchy-Schwarz inequality together yield that $\|A\|_{\ttinf} \le \underset{i \in [p_1]}{\textnormal{max}}\|A_{i}\|_{2}$, since
		\begin{equation*}
			\|A\|_{\ttinf}
			:=\underset{\|x\|_2=1}{\text{sup}}\|Ax\|_{\infty}
			=\underset{\|x\|_{2}=1}{\text{sup}} \underset{i \in [p_1]}{\text{max}}\left| \langle Ax,e_{i}\rangle\right|
			\le \underset{i \in [p_1]}{\textnormal{max}}\|A_{i}\|_{2}.
		\end{equation*}
		Barring the trivial case $A\equiv0$, let $e_{\star}$ denote the standard basis vector in $\R^{p_{1}}$ with index given by $\text{arg max}_{i \in [p_{1}]}\|A_{i}\|_{2} > 0$, noting that for each $i \in [p_{1}]$, $A_{i} = e_{i}^{\top}A$. Then for the unit Euclidean norm vector $x_{\star}:=\|e_{\star}^{\top}A\|_{2}^{-1}(e_{\star}^{\top}A)$,
		\begin{equation*}
			\|A\|_{\ttinf}
			= \underset{\|x\|_{2}=1}{\text{sup}} \underset{i \in [p_1]}{\text{max}}\left| \langle Ax,e_{i}\rangle\right|
			\ge \left| \langle A x_{\star},e_{\star}\rangle\right|
			=\|e_{\star}^{\top}A\|_{2}=\underset{i \in [p_1]}{\textnormal{max}}\|A_{i}\|_{2}.
		\end{equation*}
		This establishes the stated equivalence.
	\end{proof}
	\begin{remark}
		The two-to-infinity norm is subordinate with respect to the $\ell_{2}$ and $\ell_{\infty}$ vector norms in the sense that for any $x \in \R^{p_{2}}$, $\|Ax\|_{\infty}\le\|A\|_{\ttinf}\|x\|_{2}$. However, $\|\cdot\|_{\ttinf}$ is not sub-multiplicative for matrices in general. For example,
		$\|AB\|_{\ttinf}
		= \sqrt{5}
		> \sqrt{4}
		= \|A\|_{\ttinf}\|B\|_{\ttinf}$
		when
		\begin{align*}
			A \equiv B := \begin{bmatrix}
				1 & 1 \\
				0 & 1
			\end{bmatrix}
			\text{ and }
			AB = \begin{bmatrix}
				1 & 2 \\
				0 & 1
			\end{bmatrix}.
		\end{align*}
	\end{remark}
	\begin{proposition}
		\label{prop:2infty_&_spec}
		For $A \in \R^{p_{1} \times p_{2}}$, then
		\begin{equation}
			\|A\|_{\ttinf}
			\le \|A\|_{2}
			\le \textnormal{min}\{\sqrt{p_1} \|A\|_{\ttinf}, \sqrt{p_{2}} \|A^{\top}\|_{\ttinf}\}.
		\end{equation}
	\end{proposition}
	\begin{proof}
		The first inequality is obvious since
		\begin{equation*}
			\|A\|_{\ttinf}
			= \underset{\|x\|_{2}=1}{\text{sup}} \underset{i \in [p_1]}{\text{max}}\left| \langle Ax,e_{i}\rangle\right|
			\le \underset{\|x\|_{2}=1}{\text{sup}} \underset{\|y\|_{2}=1}{\text{sup}}\left| \langle Ax,y\rangle\right|
			= \|A\|_{2}.
		\end{equation*}
		The second inequality holds by an application of the Cauchy-Schwarz inequality together with the vector norm relationship $\|Ax\|_{2} \le \sqrt{p_{1}}\|Ax\|_{\infty}$ for $Ax \in \R^{p_{1}}$. In particular,
		\begin{equation*}
			\underset{\|x\|_{2}=1}{\text{sup}} \underset{\|y\|_{2}=1}{\text{sup}}\left| \langle Ax,y\rangle\right|
			\le \underset{\|x\|_{2}=1}{\text{sup}}\|Ax\|_{2}
			\le \sqrt{p_1}\underset{\|x\|_{2}=1}{\text{sup}}\|Ax\|_{\infty}
			= \sqrt{p_1} \|A\|_{\ttinf}.
		\end{equation*}
		By the transpose-invariance of the spectral norm, then similarly
		\begin{equation*}
			\|A\|_{2} = \|A^{\top}\|_{2} \le \sqrt{p_{2}}\|A^{\top}\|_{\ttinf}. \qedhere
		\end{equation*}
	\end{proof}
	\begin{remark}
		Proposition~\ref{prop:2infty_&_spec} is sharp. Indeed, for the second inequality, take $A := \{1/\sqrt{p_{2}}\}^{p_{1} \times p_{2}}$. Then $\|A\|_{2\rightarrow\infty}=1$ and $\|A^{\top}\|_{2\rightarrow\infty}=\sqrt{p_{1}/p_{2}}$ while $\|A\|_{2}=\sqrt{p_1}$. For ``tall, skinny'' rectangular matrices, the two-to-infinity norm can be much smaller than the spectral norm.
	\end{remark}
	\begin{proposition}
		\label{prop:2infty_relations}
		For $A \in \R^{p_{1} \times p_{2}}$, $B \in \R^{p_{2} \times p_{3}}$, and $C \in \R^{p_{4} \times p_{1}}$, then
		\begin{equation}
			\label{eqn:submult 2inf spec}
			\|AB\|_{\ttinf}
			\le \|A\|_{\ttinf}\|B\|_{2};
		\end{equation}
		\begin{equation}
			\label{eqn:submult 2inf inf}
			\|CA\|_{\ttinf} \le \|C\|_{\infty}\|A\|_{\ttinf}.
		\end{equation}
	\end{proposition}
	\begin{proof}
		The subordinate property of $\|\cdot\|_{2\rightarrow\infty}$ yields that for all $x \in \R^{p_3}$, $\|ABx\|_{\infty}\le\|A\|_{2\rightarrow\infty}\|Bx\|_{2}$, hence maximizing over all unit vectors $x$ yields Eq.~(\ref{eqn:submult 2inf spec}).
		Equation~(\ref{eqn:submult 2inf inf}) follows from H\"{o}lder's inequality coupled with the fact that the vector norms $\|\cdot\|_{1}$ and $\|\cdot\|_{\infty}$ are dual to one another. Explicitly,
		\begin{align*}
			\|CA\|_{\ttinf}
			&= \underset{\|x\|_{2}=1}{\text{sup}} \underset{i \in [p_1]}{\text{max}}\left| \langle CAx,e_{i}\rangle\right|
			\le \underset{\|x\|_{2}=1}{\text{sup}} \underset{i \in [p_1]}{\text{max}} \|C^{\top}e_{i}\|_{1} \|Ax\|_{\infty}\\
			&\le \left[\underset{\|y\|_{1}=1}{\text{sup}} \|C^{\top}y\|_{1}\right]
			\left[\underset{\|x\|_{2}=1}{\text{sup}} \|Ax\|_{\infty}\right]
			= \|C^{\top}\|_{1}\|A\|_{\ttinf}\\
			&= \|C\|_{\infty}\|A\|_{\ttinf}. &\qedhere
		\end{align*}
	\end{proof}
	\begin{proposition}
		\label{prop:partial_isometry_invariance}
		For $A\in\R^{r \times s}$, $U\in\mathbb{O}_{p_{1},r}$, and $V\in\mathbb{O}_{p_{2},s}$,
		\begin{align}
			\|A\|_{2}&=\|UA\|_{2}=\|AV^{\top}\|_{2} = \|UAV^{\top}\|_{2};\\
			\|A\|_{\ttinf}&=\|AV^{\top}\|_{\ttinf}.
		\end{align}
		However, $\|UA\|_{\ttinf}$ need not equal $\|A\|_{\ttinf}$.
		\begin{proof}
			The first statement follows from Proposition~\ref{prop:2infty_relations}, the fact that the spectral norm is sub-multiplicative, and since $U^{\top}U$, $V^{\top}V$ are both the identity matrix. As for the final claim, consider the matrices
			\begin{equation}
				U := \begin{bmatrix}
					1/\sqrt{2} & 1/\sqrt{2} \\
					1/\sqrt{2} & -1/\sqrt{2}
				\end{bmatrix},
				A := \begin{bmatrix}
					1 & 1 \\
					0 & 1
				\end{bmatrix},
				UA= \begin{bmatrix}
					1/\sqrt{2} & \sqrt{2} \\
					1/\sqrt{2} & 0
				\end{bmatrix},
			\end{equation}
			for which $\|UA\|_{\ttinf} = \sqrt{5/2} > \sqrt{2} = \|A\|_{\ttinf}$.
		\end{proof}
	\end{proposition}	
	%%%%%%%%%%%%%%%%%%%%%%%%%%%%%%%%%%%%%%%%%%%%%%%%%%%
	%%%%%%%%%%%%%%%%%%%%%%%%%%%%%%%%%%%%%%%%%%%%%%%%%%%
	\subsection{Singular subspace geometric bounds}
	\label{sec:Geometric_bounds}
	Let $U,\hat{U}\in\mathbb{O}_{p \times r}$ and $W_{U}\in\mathbb{O}_{r}$ denote the Frobenius-optimal Procrustes transformation. We shall use the fact that $\|\sin\Theta(\hat{U},U)\|_{2}=\|U_{\perp}^{\top}\hat{U}\|_{2}=\|(I-UU^{\top})\hat{U}\hat{U}^{\top}\|_{2}$ (\cite{Bhatia--1997}, Chapter 7).
	\begin{lemma}
		\label{lemma:geom_resid_bounds}
		Let $T\in\R^{r \times r}$ be arbitrary. Then,
		\begin{align}
			\|\sin \Theta(\hat{U},U)\|_{2}
			&= \|\hat{U} - UU^{\top}\hat{U}\|_{2}
			\le \|\hat{U}-UT\|_{2},\\
			\tfrac{1}{2}\|\sin \Theta(\hat{U},U) \|_{2}^{2}
			&\le \|U^{\top}\hat{U}-W_{U}\|_{2}
			\le \|\sin \Theta(\hat{U},U) \|_{2}^{2}.
		\end{align}
	\end{lemma}
	\begin{proof}
		The matrix difference $(\hat{U} - UU^{\top}\hat{U}) \in \R^{p \times r}$ represents the residual of $\hat{U}$ after orthogonally projecting onto the subspace spanned by the columns of $U$. Note that $\|A\|_{2}^{2}=\|A^{\top}A\|_{2}=\text{sup}_{\|x\|_{2}=1}|\langle A^{\top}Ax,x\rangle|$, and so several intermediate steps of computation yield that for any $T \in \R^{r \times r}$,
		\begin{align*}
			\|\hat{U} - UU^{\top}\hat{U}\|_{2}^{2}
			&=
			\text{sup}_{\|x\|_{2}=1}|\langle(\hat{U} - UU^{\top}\hat{U})^{\top}(\hat{U} - UU^{\top}\hat{U})x,x\rangle|\\
			&= \text{sup}_{\|x\|_{2}=1}|\langle(I-\hat{U}^{\top}U U^{\top}\hat{U})x,x\rangle|\\
			&\le \text{sup}_{\|x\|_{2}=1}\left(|\langle(I-\hat{U}^{\top}U U^{\top}\hat{U})x,x\rangle| + \|(T - U^{\top}\hat{U})x\|_{2}^{2}\right)\\
			&= \text{sup}_{\|x\|_{2}=1}|\langle(\hat{U} - U T)^{\top}(\hat{U} - U T)x,x\rangle|\\
			&= \|\hat{U} - U T\|_{2}^{2}.
		\end{align*}
		On the other hand, by Proposition~\ref{prop:partial_isometry_invariance} and the above observation,
		\begin{equation*}
			\|\hat{U} - UU^{\top}\hat{U}\|_{2}
			=\|\hat{U}\hat{U}^{\top} - UU^{\top}\hat{U}\hat{U}^{\top}\|_{2}
			=\|(I-UU^{\top})\hat{U}\hat{U}^{\top}\|_{2}
			=\|\sin \Theta(\hat{U},U)\|_{2}.
		\end{equation*}
		The matrix difference $(U^{\top}\hat{U}-W_{U})\in\R^{r \times r}$ represents the extent to which $U^{\top}\hat{U}$ with singular value decomposition $W_{1}\Sigma_{U}W_{2}^{\top}$ is ``almost'' the Frobenius-optimal Procrustes transformation $W_{U}\equiv W_{1}W_{2}^{\top}$. The orthogonal invariance of the spectral norm together with the interpretation of canonical angles between $\hat{U}$ and $U$, denoted by $\{\theta_{i}\}$ with $\cos(\theta_{i})=\sigma_{i}(U^{\top}\hat{U}) \in [0,1]$, yields
		\begin{equation*}
			\|U^{\top}\hat{U}-W_{U}\|_{2}
			= \|W_{1}\Sigma_{U}W_{2}^{\top} - W_{1}W_{2}^{\top}\|_{2}
			= \|\Sigma_{U}-I_{r}\|_{2} 
			= 1 - \text{min}_{i}\cos(\theta_{i}).
		\end{equation*}
		Thus, both
		\begin{equation*}
			\|U^{\top}\hat{U}-W_{U}\|_{2}
			\le 1 - \text{min}_{i}\cos^{2}(\theta_{i})
			= \text{max}_{i}\sin^{2}(\theta_{i})
			= \|\sin \Theta(\hat{U},U) \|_{2}^{2}
		\end{equation*}
		and 
		\begin{equation*}
			\|U^{\top}\hat{U}-W_{U}\|_{2}
			\ge \tfrac{1}{2}(1 - \text{min}_{i}\cos^{2}(\theta_{i}))
			= \tfrac{1}{2}\text{max}_{i}\sin^{2}(\theta_{i})
			= \tfrac{1}{2}\|\sin \Theta(\hat{U},U) \|_{2}^{2}. \qedhere
		\end{equation*}	
	\end{proof}
	\begin{lemma}
		\label{lem:Frob_opt_in_spectral}
		The quantity $\|\hat{U}-UW_{U}\|_{2}$ satisfies the lower bound
		\begin{equation}
			\|\sin\Theta(\hat{U},U)\|_{2}
			\le \|\hat{U}-UW_{2}^{\star}\|_{2}
			\le \|\hat{U}-UW_{U}\|_{2}
		\end{equation}
		and satisfies the upper bound
		\begin{equation}
			\|\hat{U}-UW_{U}\|_{2}
			\le \|\sin\Theta(\hat{U},U)\|_{2} + \|\sin\Theta(\hat{U},U)\|_{2}^{2}.
		\end{equation}
		Taken together with Lemma 1 in \cite{Cai-Zhang--2016}, an improved upper bound is given by
		\begin{equation}
			\|\hat{U}-UW_{U}\|_{2}
			\le \textnormal{min}\{1+\|\sin\Theta(\hat{U},U)\|_{2}, \sqrt{2}\}\|\sin\Theta(\hat{U},U)\|_{2}.
		\end{equation}
	\end{lemma}
	\begin{proof}
		The lower bound follows from setting $T=W_{2}^{\star}$ in Lemma~\ref{lemma:geom_resid_bounds} together with the definition of $W_{2}^{\star}$. Again by Lemma~\ref{lemma:geom_resid_bounds} and together with the triangle inequality,
		\begin{align*}
			\|\hat{U}-UW_{U}\|_{2}
			&\le \|\hat{U}-UU^{\top}\hat{U}\|_{2} + \|U(U^{\top}\hat{U}-W_{U})\|_{2}\\
			&\le \|\sin\Theta(\hat{U},U)\|_{2} + \|\sin\Theta(\hat{U},U)\|_{2}^{2}.
		\end{align*}
		The proof of Lemma 1 in \cite{Cai-Zhang--2016} establishes that
		\begin{equation*}
			\underset{W\in\mathbb{O}_{r}}{\textnormal{inf}}\|\hat{U}-UW\|_{2}
			\le \|\hat{U}-UW_{U}\|_{2}
			\le \sqrt{2}\|\sin\Theta(\hat{U},U)\|_{2},
		\end{equation*}
		which completes the proof.
	\end{proof}	
	%%%%%%%%%%%%%%%%%%%%%%%%%%%%%%%%%%%%%%%%%%%%%%%%%%%
	%%%%%%%%%%%%%%%%%%%%%%%%%%%%%%%%%%%%%%%%%%%%%%%%%%%
	For ease of reference and notation, Theorem~\ref{thrm:mod_Samworth} below states a version of the Davis-Kahan $\sin \Theta$ theorem \cite{Davis-Kahan--1970} in the language of \cite{Yu-Wang-Samworth--2015}. This amounts to a recasting of Theorem VII.3.2 in \cite{Bhatia--1997}, and so we omit the proof.
	\begin{theorem}
		\label{thrm:mod_Samworth}
		Let $X,\hat{X}\in\R^{p \times p}$ be symmetric matrices with eigenvalues $\lambda_{1} \ge \dots \ge \lambda_{p}$ and $\hat{\lambda}_{1} \ge \dots \ge \hat{\lambda}_{p}$, respectively. Write $E:=\hat{X}-X$ and fix $1 \le r \le s \le p$. Assume that $\delta_{\emph{\text{gap}}}:=\emph{\text{min}}(\lambda_{r-1}-\lambda_{r},\lambda_{s}-\lambda_{s+1})>0$ where $\lambda_{0}:=\infty$ and $\lambda_{p+1}:=-\infty$. Let $d=s-r+1$ and let $V:=[v_{r}|v_{r+1}|\dots|v_{s}]\in\R^{p \times d}$ and $\hat{V}:=[\hat{v}_{r}|\hat{v}_{r+1}|\dots|\hat{v}_{s}]\in\R^{p \times d}$ have orthonormal columns satisfying $X v_{j}=\lambda_{j}v_{j}$ and $\hat{X}\hat{v}_{j}=\hat{\lambda}_{j}\hat{v}_{j}$ for $j=r,r+1,\dots,s.$ Then
		\begin{equation}
			\|\sin\Theta(\hat{V},V)\|_{2} \le \left(\frac{2\|E\|_{2}}{\delta_{\emph{\text{gap}}}}\right).
		\end{equation}
	\end{theorem}
	%%%%%%%%%%%%%%%%%%%%%%%%%%%%%%%%%%%%%%%%%%%%%%%%%%%
	%%%%%%%%%%%%%%%%%%%%%%%%%%%%%%%%%%%%%%%%%%%%%%%%%%%

\subsection{Theorem~\ref{thrm:rectangular_decomp}}
\label{sec:rect_matrix_decomp}
\begin{proof}[Proof of Theorem~\ref{thrm:rectangular_decomp}]
	The matrices $X$ and $\hat{X}$ have rank at least $r$, so $\hat{U} \equiv \hat{X}\hat{V}\hat{\Sigma}^{-1}$ and $U W_{U} \equiv X V \Sigma^{-1} W_{U}$ by the block matrix formulation in Section~\ref{subsec:SVD_perturb_framework}. The explicit correspondence between $W_{U}$ and $U^{\top}\hat{U}$ along with subsequent left-multiplication by the matrix $U$ motivates the introduction of the projected quantity $\pm UU^{\top}\hat{U}$ and leads to
	\begin{align*}
		\hat{U} - U W_{U}
		&= (\hat{U} - UU^{\top}\hat{U}) + (UU^{\top}\hat{U} - U W_{U})\\
		&= (I-UU^{\top})\hat{X}\hat{V}\hat{\Sigma}^{-1} + U(U^{\top}\hat{U} - W_{U}).
	\end{align*}
	The matrix $U(U^{\top}\hat{U} - W_{U})$ can be meaningfully bounded in both spectral and two-to-infinity norm by Lemma~\ref{lem:Frob_opt_in_spectral} and Proposition~\ref{prop:2infty_relations}. Ignoring $U$ for the moment, the difference $U^{\top}\hat{U} - W_{U}$ represents the geometric approximation error between $U^{\top}\hat{U}$ and the orthogonal matrix $W_{U}$.
	
	It is not immediately clear how to control $(I-UU^{\top})\hat{X}\hat{V}\hat{\Sigma}^{-1}$ given the dependence on the perturbed quantity $\hat{X}$. If instead we replace $\hat{X}$ with $X$ and consider the matrix $(I-UU^{\top})X\hat{V}\hat{\Sigma}^{-1}$, then by the block matrix form in Section \ref{subsec:SVD_perturb_framework} one can check that $(I-UU^{\top})X = X(I-VV^{\top})$. Since $(I-UU^{\top})$ is an orthogonal projection and hence is idempotent, then $(I-UU^{\top})X\hat{V}\hat{\Sigma}^{-1}
	= (I-UU^{\top})X(\hat{V}-VV^{\top}\hat{V})\hat{\Sigma}^{-1}$.
	Therefore,
	\begin{align*}
		(I-UU^{\top})\hat{X}\hat{V}\hat{\Sigma}^{-1}
		&= (I-UU^{\top})E\hat{V}\hat{\Sigma}^{-1} + (I-UU^{\top})X(\hat{V}-VV^{\top}\hat{V})\hat{\Sigma}^{-1}.
	\end{align*}
	By Lemma~\ref{lemma:geom_resid_bounds} and Proposition~\ref{prop:2infty_relations}, the terms comprising the matrix product $(I-UU^{\top})X(\hat{V}-VV^{\top}\hat{V})\hat{\Sigma}^{-1}$ can be suitably controlled in norm. At times it shall be useful to further decompose $(I-UU^{\top})X(\hat{V}-V V^{\top} \hat{V})\hat{\Sigma}^{-1}$ as
	\begin{equation*}
		\label{eqn:third_matrix_decomp}
		\left((I-UU^{\top})X(\hat{V}-VW_{V})\hat{\Sigma}^{-1}\right)
		+ \left((I-UU^{\top})XV(W_{V}-V^{\top}\hat{V})\hat{\Sigma}^{-1}\right),
	\end{equation*} 
	where the second term vanishes since $U_{\perp}U_{\perp}^{\top}XV$ vanishes.
	
	As for the matrix $(I-UU^{\top})E\hat{V}\hat{\Sigma}^{-1}$, we do not assume explicit control of $\hat{V}$, so we rewrite the above matrix product in terms of $V$ and a corresponding residual quantity. A natural choice is to incorporate the orthogonal factor $W_{V}$. Specifically, introducing $\pm (I-UU^{\top})EVW_{V}\hat{\Sigma}^{-1}$ yields	
	\begin{align*}
		(I-UU^{\top})E\hat{V}\hat{\Sigma}^{-1}
		= (I-UU^{\top})E(\hat{V}-VW_{V})\hat{\Sigma}^{-1} + (I-UU^{\top})EVW_{V}\hat{\Sigma}^{-1}.
	\end{align*}
	Gathering right-hand sides of the above equations yields Theorem~\ref{thrm:rectangular_decomp}. Corollaries~\ref{cor:symmetric_decomp}~and~\ref{cor:rectangular_rewritten} are evident given that $U^{\top}U$ and $V^{\top}V$ are both simply the identity matrix. Corollary~\ref{cor:rectangular_expanded} is obtained from Corollary~\ref{cor:rectangular_rewritten} by additional straightforward algebraic manipulations. In applications, $(I-UU^{\top})EVW_{V}\hat{\Sigma}^{-1}$ can be shown to function as the leading order term.
\end{proof}

%%%%%%%%%%%%%%%%%%%%%%%%%%%%%%%%%%%%%%%%%%%%%%%%%%%
%%%%%%%%%%%%%%%%%%%%%%%%%%%%%%%%%%%%%%%%%%%%%%%%%%%
\subsection{Theorem~\ref{thrm:baselineProcrustesPerturbBound}}
\begin{proof}[Proof of Theorem~\ref{thrm:baselineProcrustesPerturbBound}]
	The assumption $\sigma_{r}(X) \ge 2\|E\|_{2}$ implies that $\sigma_{r}(\hat{X}) \ge \frac{1}{2}\sigma_{r}(X)$ since by Weyl's inequality for singular values, $\sigma_{r}(\hat{X})\ge\sigma_{r}(X)-\|E\|_{2}\ge\frac{1}{2}\sigma_{r}(X)$. The result then follows from Corollary~\ref{cor:rectangular_expanded} together with Proposition~\ref{prop:2infty_relations} and Lemma~\ref{lemma:geom_resid_bounds}.
\end{proof}
%%%%%%%%%%%%%%%%%%%%%%%%%%%%%%%%%%%%%%%%%%%%%%%%%%%
\subsection{Theorem~\ref{thrm:General_rectangular_bounds}}
\begin{proof}[Proof of Theorem~\ref{thrm:General_rectangular_bounds}]
	By Corollary~\ref{cor:rectangular_rewritten}, consider the decomposition
	\begin{align*}
		\hat{U}-UW_{U} &= (I-UU^{\top})E(VV^{\top})VW_{V}\hat{\Sigma}^{-1}\\
		&+ (I-UU^{\top})(E+X)(\hat{V}-VW_{V})\hat{\Sigma}^{-1}\\
		&+ U(U^{\top}\hat{U}-W_{U}).
	\end{align*}
	Applying Proposition~\ref{prop:2infty_relations} and Lemma~\ref{lemma:geom_resid_bounds} then yields
	\begin{align*}
		\|\hat{U}-UW_{U}\|_{2\rightarrow\infty}
		&\le \|(U_{\perp}U_{\perp}^{\top})E(VV^{\top})\|_{\ttinf}/\sigma_{r}(\hat{X})\\
		&+ \left(C_{E,U}+C_{X,U}\right)\|\hat{V}-VW_{V}\|_{\ttinf}/\sigma_{r}(\hat{X})\\
		&+ \|\sin\Theta(\hat{U},U)\|_{2}^{2}\|U\|_{\ttinf}
	\end{align*}
	and similarly
	\begin{align*}
		\|\hat{V}-VW_{V}\|_{\ttinf}
		&\le \|(V_{\perp}V_{\perp}^{\top})E^{\top}(UU^{\top})\|_{\ttinf}/\sigma_{r}(\hat{X})\\
		&+ \left(C_{E,V}+C_{X,V}\right)\|\hat{U}-UW_{U}\|_{\ttinf}/\sigma_{r}(\hat{X})\\
		&+ \|\sin\Theta(\hat{V},V)\|_{2}^{2}\|V\|_{\ttinf}.
	\end{align*}
	By assumption,
	\begin{equation*}
		\sigma_{r}(X) \ge \text{max}\{2\|E\|_{2},(2/\alpha)C_{E,U},(2/\alpha^{\prime})C_{E,V},(2/\beta)C_{X,U},(2/\beta^{\prime})C_{X,V}\}
	\end{equation*}
	for some constants $0<\alpha,\alpha^{\prime},\beta,\beta^{\prime}<1$ such that $\delta:=(\alpha+\beta)(\alpha^{\prime}+\beta^{\prime})<1$. The assumption $\sigma_{r}(X) \ge 2\|E\|_{2}$ implies that $\sigma_{r}(\hat{X})\ge\sigma_{r}(X)-\|E\|_{2}\ge\frac{1}{2}\sigma_{r}(X)$ by Weyl's inequality for singular values. Combining the above observations and rearranging terms yields
	\begin{align*}
		(1-\delta)\|\hat{U}-UW_{U}\|_{2\rightarrow\infty}
		&\le 2\|(U_{\perp}U_{\perp}^{\top})E(VV^{\top})\|_{\ttinf}/\sigma_{r}(X)\\
		&+ 2(\alpha+\beta)\|(V_{\perp}V_{\perp}^{\top})E^{\top}(UU^{\top})\|_{\ttinf}/\sigma_{r}(X) \\
		&+ \|\sin\Theta(\hat{U},U)\|_{2}^{2}\|U\|_{\ttinf} \\
		&+ \left(\alpha+\beta\right)\|\sin\Theta(\hat{V},V)\|_{2}^{2}\|V\|_{\ttinf}.
	\end{align*}
	The first claim follows since $(\alpha+\beta)<1$. When $\textnormal{rank}(X)=r$, then $U_{\perp}U_{\perp}^{\top}X$ vanishes. Corollary~\ref{cor:symmetric_decomp} then yields the simpler form
	\begin{align*}
		\hat{U}-UW_{U} &= (I-UU^{\top})E(VV^{\top})VW_{V}\hat{\Sigma}^{-1}\\
		&+ (I-UU^{\top})E(\hat{V}-VW_{V})\hat{\Sigma}^{-1}\\
		&+ U(U^{\top}\hat{U}-W_{U}),
	\end{align*}
	and similarly for $\hat{V}-VW_{V}$. In this case, the bound holds without needing to consider either $C_{X,U}$ or $C_{X,V}$.
\end{proof}
%%%%%%%%%%%%%%%%%%%%%%%%%%%%%%%%%%%%%%%%%%%%%%%%%%%
\subsection{Corollary~\ref{cor:reFanThrmRect}}
\begin{proof}[Proof of Corollary~\ref{cor:reFanThrmRect}]
	By Theorem~\ref{thrm:General_rectangular_bounds},
	\begin{align*}
		(1-\delta)\|\hat{U}-UW_{U}\|_{\ttinf}
		&\le 2\|(U_{\perp}U_{\perp}^{\top})E(VV^{\top})\|_{\ttinf}/\sigma_{r}(X)\\
		&+ 2\|(V_{\perp}V_{\perp}^{\top})E^{\top}(UU^{\top})\|_{\ttinf}/\sigma_{r}(X) \\
		&+ \|\sin\Theta(\hat{U},U)\|_{2}^{2}\|U\|_{2\rightarrow\infty} \\
		&+ \|\sin\Theta(\hat{V},V)\|_{2}^{2}\|V\|_{2\rightarrow\infty}.
	\end{align*}
	Applying Wedin's $\sin\Theta$ theorem together with the general matrix fact that $\|E\|_{2} \le \text{max}\{\|E\|_{\infty},\|E\|_{1}\}$ and the assumption $\sigma_{r}(X) \ge 2\|E\|_{2}$ yields
	\begin{equation*}
		\underset{Z \in \{U,V\}}{\textnormal{ max}}\left\{\|\sin\Theta(\hat{Z},Z)\|_{2}\right\}
		\le \textnormal{min}\left\{
		\left(\frac{2\times\text{max}\{\|E\|_{\infty},\|E\|_{1}\}}{\sigma_{r}(X)}\right),
		1
		\right\}.
	\end{equation*}
	Using properties of the two-to-infinity norm, then
	\begin{align*}
		\|(U_{\perp}U_{\perp}^{\top})E(VV^{\top})\|_{\ttinf}
		&\le \|E(VV^{\top})\|_{\ttinf} + \|(UU^{\top})E(VV^{\top})\|_{\ttinf}\\
		&\le \|EV\|_{\ttinf} + \|U\|_{\ttinf}\|U^{\top}EV\|_{2}\\
		&\le \|E\|_{\infty}\|V\|_{\ttinf} + \|U\|_{\ttinf}\text{max}\{\|E\|_{\infty},\|E\|_{1}\}\\
		&\le 2 \times \underset{\eta \in \{1,\infty\}}{\textnormal{max}}\left\{\|E\|_{\eta}\right\}\times
		\underset{Z\in\{U,V\}}{\textnormal{max}}\left\{\|Z\|_{\ttinf}\right\}.
	\end{align*}
	Similarly,
	\begin{align*}
		\|(V_{\perp}V_{\perp}^{\top})E^{\top}(UU^{\top})\|_{\ttinf}
		&\le \|E\|_{1}\|U\|_{\ttinf} + \|V\|_{\ttinf}\text{max}\{\|E\|_{\infty},\|E\|_{1}\}\\
		&\le 2 \times \underset{\eta \in \{1,\infty\}}{\textnormal{max}}\left\{\|E\|_{\eta}\right\}\times
		\underset{Z\in\{U,V\}}{\textnormal{max}}\left\{\|Z\|_{\ttinf}\right\}.
	\end{align*}
	Hence
	$
	(1-\delta)\|\hat{U}-UW_{U}\|_{\ttinf}
	\le 12
	\times
	\underset{\eta \in \{1, \infty\}}{\textnormal{max}}\left\{\frac{\|E\|_{\eta}}{\sigma_{r}(X)}\right\}\times
	\underset{Z \in \{U, V\}}{\textnormal{max}}\left\{\|Z\|_{\ttinf}\right\}
	$.
\end{proof}
%%%%%%%%%%%%%%%%%%%%%%%%%%%%%%%%%%%%%%%%%%%%%%%%%%%
%%%%%%%%%%%%%%%%%%%%%%%%%%%%%%%%%%%%%%%%%%%%%%%%%%%
\subsection{Theorem~\ref{thrm:generalCovarEst}}
\label{sec:covarianceProof}
\begin{proof}[Proof of Theorem~\ref{thrm:generalCovarEst}]
	For ease of presentation, we use $C>0$ to denote various constants that are allowed to depend on one another. Both $n$ and $d$ are taken to be large.
	
	By hypothesis $\rmMax\{\mathfrak{r}(\Gamma), \log d\} = o(n)$, where $\mathfrak{r}(\Gamma) := \textnormal{trace}(\Gamma)/ \sigma_{1}(\Gamma)$ denotes the effective rank of $\Gamma$. In the present multivariate Gaussian covariance matrix setting, it follows from \cite{koltchinskii2017concentration,koltchinskii2017pca} that there exists some constant $C>0$ such that with probability at least $1 - \tfrac{1}{3}d^{-2}$,
	\begin{equation*}
		\|E_{n}\|_{2} \le C \sigma_{1}(\Gamma)\sqrt{\tfrac{\textrm{max}\{\mathfrak{r}(\Gamma), \log d\}}{n}}.
	\end{equation*}
	By hypothesis $\sigma_{1}(\Gamma) /\sigma_{r}(\Gamma)\le C$, and so together with the above observations, then $\sigma_{r}(\Gamma) \ge 2\|E_{n}\|_{2}$ with high probability. Theorem~\ref{thrm:baselineProcrustesPerturbBound} thus yields
	\begin{align*}
		\|\hat{U} - UW_{U}\|_{\ttinf}
		&\le C\|(U_{\perp}U_{\perp}^{\top})E_{n}(UU^{\top})\|_{\ttinf}/\sigma_{r}(\Gamma)\\
		&+ C\|(U_{\perp}U_{\perp}^{\top})E_{n}(U_{\perp}U_{\perp}^{\top})\|_{\ttinf}\|\sin\Theta(\hat{U},U)\|_{2}/\sigma_{r}(\Gamma)\\
		&+ C\|(U_{\perp}U_{\perp}^{\top})\Gamma(U_{\perp}U_{\perp}^{\top})\|_{\ttinf}\|\sin\Theta(\hat{U},U)\|_{2}/\sigma_{r}(\Gamma)\\
		&+ \|\sin\Theta(\hat{U},U)\|_{2}^{2}\|U\|_{\ttinf}.
	\end{align*}
	Moving forward, we record several important observations.
	\begin{itemize}
		\item By Proposition~\ref{prop:2infty_relations}, $\|(U_{\perp}U_{\perp}^{\top})E_{n}(UU^{\top})\|_{\ttinf}
		\le \|U_{\perp}U_{\perp}^{\top}\|_{\infty}\|E_{n}U\|_{\ttinf}$.
		\item By the (bounded coherence) assumption that $\|U\|_{\ttinf}\le C\sqrt{r/d}$, then
		\begin{equation*}
			\|U_{\perp}U_{\perp}^{\top}\|_{\infty}
			= \|I - UU^{\top}\|_{\infty}
			\le 1 + \sqrt{d}\|UU^{\top}\|_{\ttinf}
			\le(1+C)\sqrt{r}.
			\end{equation*}
		\item The random (Gaussian) vector $U_{\perp}^{\top}Y$ has covariance matrix $U_{\perp}\Sigma_{\perp}U_{\perp}^{\top}$, so by \cite{koltchinskii2017concentration,koltchinskii2017pca} there exists some constant $C>0$ such that with probability at least $1-\tfrac{1}{3}d^{-2}$,
		\begin{align*}
			\|(U_{\perp}U_{\perp}^{\top})E_{n}(U_{\perp}U_{\perp}^{\top})\|_{2}
			&\le C \sigma_{r+1}(\Gamma)\sqrt{\tfrac{\textrm{max}\{\mathfrak{r}(\Sigma_{\perp}), \log d\}}{n}}\\
			&\le C \sqrt{\sigma_{r+1}(\Gamma)}\sqrt{\sigma_{1}(\Gamma)}\sqrt{\tfrac{\textrm{max}\{\mathfrak{r}(\Gamma), \log d\}}{n}},
			\end{align*}
		where the final inequality holds since
		\begin{equation*}
			\mathfrak{r}(\Sigma_{\perp})
			= \left[\tfrac{\sigma_{1}(\Gamma)}{\sigma_{r+1}(\Gamma)}\right]\left[\mathfrak{r}(\Gamma)-\tfrac{\rmTr(\Sigma)}{\sigma_{1}(\Gamma)}\right]
			\le \left[\tfrac{\sigma_{1}(\Gamma)}{\sigma_{r+1}(\Gamma)}\right]\mathfrak{r}(\Gamma).
		\end{equation*}
		\item Note that $\|(U_{\perp}U_{\perp}^{\top})\Gamma(U_{\perp}U_{\perp}^{\top})\|_{\ttinf} = \|U_{\perp}\Sigma_{\perp}U_{\perp}^{\top}\|_{\ttinf} \le \sigma_{r+1}(\Gamma)$.
		\item Theorem~\ref{thrm:mod_Samworth} yields the bound
		$\|\sin\Theta(\hat{U},U)\|_{2} \le C\|E_{n}\|_{2}/\delta_{r}(\Gamma)$ with population gap given by $\delta_{r}(\Gamma) := \sigma_{r}(\Gamma) - \sigma_{r+1}(\Gamma) \ge c_{2}\sigma_{r}(\Gamma)> 0$.
	\end{itemize}
	Together, these observations yield
	\begin{align*}
		\|\hat{U}-UW_{U}\|_{\ttinf}
		&\le C\sqrt{r}\|E_{n}U\|_{\ttinf}/\sigma_{r}(\Gamma)\\
		&+ C\left(\tfrac{\textrm{max}\{\mathfrak{r}(\Gamma), \log d\}}{n}\right)\sqrt{\sigma_{r+1}(\Gamma)/\sigma_{r}(\Gamma)}\\
		&+ C\sqrt{\tfrac{\textrm{max}\{\mathfrak{r}(\Gamma), \log d\}}{n}}\left(\sigma_{r+1}(\Gamma)/\sigma_{r}(\Gamma)\right)\\
		&+ C\left(\tfrac{\textrm{max}\{\mathfrak{r}(\Gamma), \log d\}}{n}\right)\sqrt{r/d}.
	\end{align*}
	Now let $e_{i}$ denote the $i$th standard basis vector in $\R^{d}$ and $u_{j}$ denote the $j$th column of $U$. The matrix $E_{n}$ is symmetric, and $E_{n}U \in \R^{d \times r}$ can be bounded in two-to-infinity norm as
	\begin{equation*}
		\|E_{n}U\|_{\ttinf}
		\le \sqrt{r}\|E_{n}U\|_{\textrm{max}}
		= \sqrt{r}\underset{i \in [d], j \in [r]}{\textrm{max}}|\langle E_{n}e_{i},u_{j}\rangle|.
	\end{equation*}
	For each $(i,j) \in [d] \times [r]$, the scalar $\langle E_{n}e_{i}, u_{j}\rangle$ can be expanded as
	\begin{equation*}
		\langle E_{n}e_{i}, u_{j}\rangle
		= \frac{1}{n}\sum_{k=1}^{n}\left[(u_{j}^{\top}Y_{k})(Y_{k}^{\top}e_{i}) - u_{j}^{\top}\Gamma e_{i}\right]
		= \frac{1}{n}\sum_{k=1}^{n}\left[\langle Y_{k},u_{j}\rangle Y_{k}^{(i)} - \langle \Gamma e_{i}, u_{j} \rangle \right].
	\end{equation*}
	The product of (sub-)Gaussian random variables is sub-exponential, so for $i$ and $j$ fixed, $[\langle Y_{k}, u_{j} \rangle Y_{k}^{(i)} - \langle \Gamma e_{i}, u_{j} \rangle]$ with $1 \le k \le n$ is a collection of independent and identically distributed, centered sub-exponential random variables \cite{vershynin2018high}. To this end, the (univariate) sub-exponential norm, the (univariate) sub-Gaussian norm, and the vector sub-Gaussian norm are respectively
	\begin{align*}
		\|(Y^{(i)})^{2}\|_{\psi_{1}}
		&:= \underset{p \ge 1}{\textrm{ sup }} p^{-1}(\mathbb{E}[|(Y^{(i)})^{2}|^{p}])^{1/p};\\
		\|Y^{(i)}\|_{\psi_{2}}
		&:= \underset{p \ge 1}{\textrm{ sup }} p^{-1/2}(\mathbb{E}[|Y^{(i)}|^{p}])^{1/p};\\
		\|Y\|_{\psi_{2}}
		&:= \underset{\|x\|_{2}=1}{\textrm{ sup }}\|\langle Y, x \rangle \|_{\psi_{2}}.
	\end{align*}
	By properties of these (Orlicz) norms \cite{vershynin2018high}, it follows that there exists some constant $C>0$ such that the above sub-exponential random variables satisfy the bound
	\begin{equation*}
		\|\langle Y_{k}, u_{j} \rangle Y_{k}^{(i)} - \langle \Gamma e_{i}, u_{j} \rangle\|_{\psi_{1}}
		\le 2\|\langle Y_{k}, u_{j} \rangle Y_{k}^{(i)}\|_{\psi_{1}}
		\le C \|\langle Y, u_{j} \rangle\|_{\psi_{2}} \|Y^{(i)}\|_{\psi_{2}}.
	\end{equation*}
	The random vector $Y$ is mean zero multivariate Gaussian, hence for each $1 \le i \le d$, the norm of the $i$th component satisfies the variance-based bound
	\begin{equation*}
		\|Y^{(i)}\|_{\psi_{2}}
		\le C \underset{1 \le i \le d}{\textrm{ max }}\sqrt{\textrm{Var}(Y^{(i)})}
		\equiv C \nu(Y).
	\end{equation*}
	For each $j \in [r]$, $\textrm{Var}(\langle Y, u_{j} \rangle) = u_{j}^{\top}\Gamma u_{j} = \sigma_{j}(\Gamma)$, where $\langle Y, u_{j} \rangle$ is univariate Gaussian, so we have the spectral-based bound $\|\langle Y, u_{j} \rangle \|_{\psi_{2}} \le C \sqrt{\sigma_{1}(\Gamma)}$. Taken together, these observations establish a uniform bound over all $i,j,k$ of the form
	\begin{equation*}
		\|\langle Y_{k}, u_{j} \rangle Y_{k}^{(i)} - \langle \Gamma e_{i}, u_{j} \rangle\|_{\psi_{1}}
		\le C \nu(Y) \sqrt{\sigma_{1}(\Gamma)}.
	\end{equation*}
	By combining a union bound with Bernstein's inequality for sub-exponential random variables \cite{vershynin2018high}, it follows that there exists a constant $C>0$ such that with probability at least $1-\tfrac{1}{3}d^{-2}$,
	\begin{align*}
		\|E_{n}U\|_{\ttinf} \le C \nu(Y) \sqrt{\sigma_{1}(\Gamma)}\sqrt{r}\sqrt{\tfrac{\textrm{max}\{\mathfrak{r}(\Gamma), \log d\}}{n}}.
	\end{align*}
	The $r$ largest singular values of $\Gamma$ bound each other up to absolute multiplicative constants for all values of $d$ by assumption. Moreover, $\delta_{r}(\Gamma) \ge c_{2}\sigma_{r}(\Gamma)$ by assumption. Aggregating the above observations yields that with probability at least $1 - d^{-2}$,
	\begin{align*}
	\|\hat{U} - UW_{U}\|_{\ttinf}
	&\le C \sqrt{\tfrac{ \textrm{max}\{\mathfrak{r}(\Gamma), \log d\}}{n}}
	\left(
	\tfrac{\nu(Y)r}{\sqrt{\sigma_{r}(\Gamma)}}
	+ \tfrac{\sigma_{r+1}(\Gamma)}{\sigma_{r}(\Gamma)} \right)\\
	&+ C \left(\tfrac{ \textrm{max}\{\mathfrak{r}(\Gamma), \log d\}}{n}\right)
	\left(\sqrt{\tfrac{\sigma_{r+1}(\Gamma)}{\sigma_{r}(\Gamma)}}
		+ \sqrt{\tfrac{r}{d}}\right),
	\end{align*}
	which completes the proof.
\end{proof}

%%%%%%%%%%%%%%%%%%%%%%%%%%%%%%%%%%%%%%%%%%%%%%%%%%%
\subsection{Theorem~\ref{thrm:FanThrm2.1++}}
\label{sec:FanThrm2.1++proof}
\begin{proof}[Proof of Theorem~\ref{thrm:FanThrm2.1++}]
	Specializing Corollary~\ref{cor:rectangular_rewritten} to the case when $X$ is symmetric with $\textnormal{rank}(X)=r$ yields the decomposition
	\begin{align*}
		\hat{U}-UW_{U}
		&= E(\hat{U}-UW_{U})\hat{\Lambda}^{-1}
		- (UU^{\top})E(\hat{U}-UW_{U})\hat{\Lambda}^{-1}
		+ EUW_{U}\hat{\Lambda}^{-1}\\
		&- (UU^{\top})EUW_{U}\hat{\Lambda}^{-1}
		+ U(U^{\top}\hat{U}-W_{U}).
	\end{align*}
	Applying the technical results in Sections~\ref{sec:Tech_tools_2infty}~and~\ref{sec:Geometric_bounds} yields the term-wise bounds
	\begin{align*}
		\|E(\hat{U}-UW_{U})\hat{\Lambda}^{-1}\|_{\ttinf}
		&\le \|E\|_{\infty}\|\hat{U}-UW_{U}\|_{2\rightarrow\infty}|\hat{\lambda}_{r}|^{-1},\\
		\|(UU^{\top})E(\hat{U}-UW_{U})\hat{\Lambda}^{-1}\|_{\ttinf}
		&\le \|U\|_{\ttinf}\|E\|_{2}\|\hat{U}-UW_{U}\|_{2}|\hat{\lambda}_{r}|^{-1},\\
		\|EUW_{U}\hat{\Lambda}^{-1}\|_{\ttinf}
		&\le \|E\|_{\infty}\|U\|_{\ttinf}|\hat{\lambda}_{r}|^{-1},\\
		\|(UU^{\top})EUW_{U}\hat{\Lambda}^{-1}\|_{\ttinf}
		&\le \|U\|_{\ttinf}\|E\|_{2}||\hat{\lambda}_{r}|^{-1},\\
		\|U(U^{\top}\hat{U}-W_{U})\|_{\ttinf}
		&\le \|U\|_{\ttinf}\|U^{\top}\hat{U}-W_{U}\|_{2}.
	\end{align*}
	By assumption $E$ is symmetric, therefore $\|E\|_{2}\le\|E\|_{\infty}$. Furthermore, $\|\hat{U}-UW_{U}\|_{2}\le\sqrt{2}\|\sin\Theta(\hat{U},U)\|_{2}$ by Lemma~\ref{lem:Frob_opt_in_spectral}, and $\|\sin\Theta(\hat{U},U)\|_{2}\le 2\|E\|_{2}|\lambda_{r}|^{-1}$ by Theorem~\ref{thrm:mod_Samworth}. Therefore,
	\begin{align*}
		\|E(\hat{U}-UW_{U})\hat{\Lambda}^{-1}\|_{2\rightarrow\infty}
		&\le \|E\|_{\infty}\|\hat{U}-UW_{U}\|_{2\rightarrow\infty}|\hat{\lambda}_{r}|^{-1},\\
		\|(UU^{\top})E(\hat{U}-UW_{U})\hat{\Lambda}^{-1}\|_{2\rightarrow\infty}
		&\le 4\|E\|_{\infty}^{2}\|U\|_{2\rightarrow\infty}|\hat{\lambda}_{r}|^{-1}|\lambda_{r}|^{-1},\\
		\|EUW_{U}\hat{\Lambda}^{-1}\|_{2\rightarrow\infty}
		&\le \|E\|_{\infty}\|U\|_{2\rightarrow\infty}|\hat{\lambda}_{r}|^{-1},\\
		\|(UU^{\top})EUW_{U}\hat{\Lambda}^{-1}\|_{2\rightarrow\infty}
		&\le \|E\|_{\infty}\|U\|_{2\rightarrow\infty}|\hat{\lambda}_{r}|^{-1},\\
		\|U(U^{\top}\hat{U}-W_{U})\|_{2\rightarrow\infty}
		&\le 4\|E\|_{\infty}^{2}\|U\|_{2\rightarrow\infty}|\lambda_{r}|^{-2}.
	\end{align*}
	By assumption $|\lambda_{r}| \ge 4\|E\|_{\infty}$, so $|\hat{\lambda}_{r}|\ge\frac{1}{2}|\lambda_{r}|$ and therefore
	$\|E\|_{\infty}|\hat{\lambda}_{r}|^{-1}
	\le 2\|E\|_{\infty}|\lambda_{r}|^{-1}
	\le \frac{1}{2}$. Thus,
	\begin{align*}
		\|E(\hat{U}-UW_{U})\hat{\Lambda}^{-1}\|_{2\rightarrow\infty}
		&\le \tfrac{1}{2}\|\hat{U}-UW_{U}\|_{2\rightarrow\infty},\\
		\|(UU^{\top})E(\hat{U}-UW_{U})\hat{\Lambda}^{-1}\|_{2\rightarrow\infty}
		&\le 2\|E\|_{\infty}\|U\|_{2\rightarrow\infty}|\lambda_{r}|^{-1},\\
		\|EUW_{U}\hat{\Lambda}^{-1}\|_{2\rightarrow\infty}
		&\le 2\|E\|_{\infty}\|U\|_{2\rightarrow\infty}|\lambda_{r}|^{-1},\\
		\|(UU^{\top})EUW_{U}\hat{\Lambda}^{-1}\|_{2\rightarrow\infty}
		&\le 2\|E\|_{\infty}\|U\|_{2\rightarrow\infty}|\lambda_{r}|^{-1},\\
		\|U(U^{\top}\hat{U}-W_{U})\|_{2\rightarrow\infty}
		&\le \|E\|_{\infty}\|U\|_{2\rightarrow\infty}|\lambda_{r}|^{-1}.
	\end{align*}
	Hence,
	$
	\|\hat{U}-UW_{U}\|_{\textrm{max}}
	\le
	\|\hat{U}-UW_{U}\|_{2\rightarrow\infty}
	\le 14\left(\frac{\|E\|_{\infty}}{|\lambda_{r}|}\right)\|U\|_{2\rightarrow\infty}.
	$
\end{proof}

%%%%%%%%%%%%%%%%%%%%%%%%%%%%%%%%%%%%%%%%%%%%%%%%%%%
\subsection{Theorem~\ref{thrm:SubspaceRecovery}}
\label{sec:subspaceProof}
\begin{proof}[Proof of Theorem~\ref{thrm:SubspaceRecovery}]
	Rewriting Corollary~\ref{cor:rectangular_rewritten} in terms of the matrix $\hat{V}-VW_{V}$ as described in Theorem~\ref{thrm:rectangular_decomp} yields the decomposition
	\begin{align*}
	\hat{V}-VW_{V}
	&=  (V_{\perp}V_{\perp}^{\top})E^{\top}(UU^{\top})UW_{U}\hat{\Sigma}^{-1}\\
	&+ (V_{\perp}V_{\perp}^{\top})(E^{\top}+X^{\top})(\hat{U}-UW_{U})\hat{\Sigma}^{-1}\\
	&+ V(V^{\top}\hat{V}-W_{V}).
	\end{align*}
	Observe that $(UU^{\top})U \equiv U$, while the assumption $\text{rank}(X)=r$ implies that $(V_{\perp}V_{\perp}^{\top})X^{\top}$ vanishes.
	Applying Proposition~\ref{prop:2infty_relations}, Lemma~\ref{lemma:geom_resid_bounds}, and Lemma~\ref{lem:Frob_opt_in_spectral} to the remaining terms therefore yields
	\begin{align*}
		\|(V_{\perp}V_{\perp}^{\top})E^{\top}UW_{U}\hat{\Sigma}^{-1}\|_{\ttinf}
		&\le \|(V_{\perp}V_{\perp}^{\top})E^{\top}U\|_{\ttinf}/\sigma_{r}(\hat{X}),\\
		\|(V_{\perp}V_{\perp}^{\top})E^{\top}(\hat{U}-UW_{U})\hat{\Sigma}^{-1}\|_{\ttinf}
		&\le C\|(V_{\perp}V_{\perp}^{\top})E^{\top}\|_{\ttinf}\|\sin\Theta(\hat{U},U)\|_{2}/\sigma_{r}(\hat{X}),\\
		\|V(V^{\top}\hat{V}-W_{V})\|_{\ttinf}
		&\le \|\sin\Theta(\hat{V},V)\|_{2}^{2}\|V\|_{\ttinf}.
	\end{align*}
	
	The columns of $(V_{\perp}V_{\perp}^{\top})E^{\top}\in\R^{p_{2}\times p_{1}}$ are centered independent multivariate normal random vectors with covariance matrix $(V_{\perp}V_{\perp}^{\top})$, so row $i$ of $(V_{\perp}V_{\perp}^{\top})E^{\top}$ is a centered multivariate normal random vector with covariance matrix $\sigma_{i}^{2}I$, where $\sigma_{i}^{2}:=(V_{\perp}V_{\perp}^{\top})_{i,i} \le 1$ and $I\in\R^{p_{1} \times p_{1}}$ denotes the identity matrix. By Gaussian concentration and applying a union bound with $p_{2} \gg p_{1}$, then
	$\|(V_{\perp}V_{\perp}^{\top})E^{\top}\|_{2\rightarrow\infty}
	\le C_{r}\sqrt{p_{1}}\log(p_{2})$ with high probability.
	
	As for $(V_{\perp}V_{\perp}^{\top})E^{\top}U\in\R^{p_{2} \times r}$, the above argument implies that entry $(i,j)$ is $\mathcal{N}(0,\sigma_{i}^{2})$. Hence by the same approach, $\|(V_{\perp}V_{\perp}^{\top})E^{\top}U\|_{2\rightarrow\infty}
	\le C_{r}\log(p_{2})$ with high probability.
	
	By hypothesis $r \ll p_{1} \ll p_{2}$ and $\sigma_{r}(X) \ge C p_{2}/\sqrt{p_{1}}$, where $\|E\|_{2} \le C\sqrt{p_{2}}$ with high probability. In this setting, via \cite{Cai-Zhang--2016},
	\begin{equation*}
		\|\sin\Theta(\hat{U},U)\|_{2}
		\le C\left(\frac{\sqrt{p_{1}}}{\sigma_{r}(X)}\right)
		\text{  and  }
		\|\sin\Theta(\hat{V},V)\|_{2}
		\le C\left(\frac{\sqrt{p_{2}}}{\sigma_{r}(X)}\right).
	\end{equation*}
	Combining these observations yields
	\begin{align*}
		\left(\tfrac{\|(V_{\perp}V_{\perp}^{\top})E^{\top}U\|_{2\rightarrow\infty}}{\sigma_{r}(\hat{X})}\right)
		&\le C_{r}\left(\tfrac{\log(p_{2})}{\sigma_{r}(X)}\right);\\
		\left(\tfrac{\|(V_{\perp}V_{\perp}^{\top})E^{\top}\|_{2\rightarrow\infty}}{\sigma_{r}(\hat{X})}\right)\|\sin\Theta(\hat{U},U)\|_{2}
		&\le C_{r}\left(\tfrac{\log(p_{2})}{\sigma_{r}(X)}\right)\left(\tfrac{p_{1}}{\sigma_{r}(X)}\right);\\
		\|\sin\Theta(\hat{V},V)\|_{2}^{2}\|V\|_{2\rightarrow\infty}
		&\le C_{r}\left(\tfrac{\log(p_{2})}{\sigma_{r}(X)}\right)\left(\tfrac{\sqrt{p_{1}}}{\log(p_{2})}\right)\|V\|_{\ttinf}.
		\end{align*}	
	Hence, with high probability
	\begin{equation*}
		\|\hat{V}-VW_{V}\|_{\ttinf}
		\le C_{r}\left(\tfrac{\log(p_{2})}{\sigma_{r}(X)}\right)
			\left(1
				+ \left(\tfrac{p_{1}}{\sigma_{r}(X)}\right)
				+ \left(\tfrac{\sqrt{p_{1}}}{\log(p_{2})}\right)\|V\|_{\ttinf}
			\right).
	\end{equation*}
	If in addition $\sigma_{r}(X) \ge c p_{1}$ and $\|V\|_{\ttinf}\le c_{r}/\sqrt{p_{2}}$ for some $c, c_{r} > 0$, then the above bound simplifies to the form
	\begin{equation*}
		\|\hat{V}-VW_{V}\|_{\ttinf}
		\le C_{r}\left(\tfrac{\log(p_{2})}{\sigma_{r}(X)}\right),
	\end{equation*}
	which completes the proof.
\end{proof}

%%%%%%%%%%%%%%%%%%%%%%%%%%%%%%%%%%%%%%%%%%%%%%%%%%%
\subsection{Theorem~\ref{thrm:rhoSBM2infinity}}
\label{sec:sbmProof}
\begin{proof}[Proof of Theorem~\ref{thrm:rhoSBM2infinity}]
	We seek to bound $\|\hat{U}-UW_{U}\|_{\ttinf}$ and allow the constant $C>0$ to change from line to line. Our analysis will consider appropriate groupings of matrix elements in order to handle the multiple graph correlation structure.
	
	By assumption $\textnormal{rank}(\mathfrak{O})=r$ which implies that the matrix $(I-UU^{\top})\mathfrak{O}$ vanishes. Via Corollary~\ref{cor:symmetric_decomp}, this yields the bound
	\begin{align*}
		\|\hat{U} - U W_{U}\|_{2\rightarrow\infty}
		&\le \|(I-UU^{\top})(\hat{\mathfrak{O}}-\mathfrak{O})U W_{U} \hat{\Sigma}^{-1}\|_{2\rightarrow\infty}\\
		&+ \|(I-UU^{\top})(\hat{\mathfrak{O}}-\mathfrak{O})(\hat{U}-UW_{U})\hat{\Sigma}^{-1}\|_{2\rightarrow\infty}\\
		&+ \|U\|_{2\rightarrow\infty}\|U^{\top}\hat{U}-W_{U}\|_{2},
	\end{align*}
	which can be further weakened to the form
	\begin{align*}
		\|\hat{U} - U W_{U}\|_{2\rightarrow\infty}
		&\le \|(\hat{\mathfrak{O}}-\mathfrak{O})U\|_{2\rightarrow\infty}\|\hat{\Sigma}^{-1}\|_{2}\\
		&+ \|U\|_{2\rightarrow\infty}\|U^{\top}(\hat{\mathfrak{O}}-\mathfrak{O})U\|_{2}\|\hat{\Sigma}^{-1}\|_{2}\\
		&+ \|\hat{\mathfrak{O}}-\mathfrak{O}\|_{2}\|\hat{U}-UW_{U}\|_{2}\|\hat{\Sigma}^{-1}\|_{2}\\
		&+ \|U\|_{2\rightarrow\infty}\|U^{\top}\hat{U}-W_{U}\|_{2}.
	\end{align*}
	Applying the triangle inequality to the block matrix $\hat{\mathfrak{O}}-\mathfrak{O}$ yields a spectral norm bound of the form
	\begin{equation*}
	\|\hat{\mathfrak{O}}-\mathfrak{O}\|_{2}
	\le C\times\textnormal{max}\{\|A^{1}-P\|_{2},\|A^{2}-P\|_{2}\}.
	\end{equation*}
	By assumption, for $i=1,2$, the maximum expected degree of $G^{i}$, $\Delta$, satisfies $\Delta \gg \log^{4}(n)$, hence $\|A^{i}-P\|_{2} = \mathcal{O}(\sqrt{\Delta})$ with probability $1-o(1)$ by \cite{Lu-Peng--2013}. The assumption $\sigma_{r}(\mathfrak{O}) \ge c \Delta$ implies that $\sigma_{r}(\hat{\mathfrak{O}}) \ge C \Delta$ with probability $1-o(1)$ by Weyl's inequality, so $\|\hat{\Sigma}^{-1}\|_{2}=\mathcal{O}(1/\Delta)$. Combining these observations with Theorem~\ref{thrm:mod_Samworth} and the proof of Lemma~\ref{lem:Frob_opt_in_spectral} yields
	\begin{align*}
		\|\hat{U}-UW_{U}\|_{2}
		\le C\|\sin\Theta(\hat{U},U)\|_{2}
		\le C\|\hat{\mathfrak{O}}-\mathfrak{O}\|_{2}/\sigma_{r}(\mathfrak{O})
		= \mathcal{O}\left(1/\sqrt{\Delta}\right),
	\end{align*}
	which we note simultaneously provides a na\"{i}ve bound for $\|\hat{U}-UW_{U}\|_{\ttinf}$.	
	As for the matrix, $(\hat{\mathfrak{O}}-\mathfrak{O})U \in \R^{2n \times r}$,
	\begin{align*}
		\|(\hat{\mathfrak{O}}-\mathfrak{O})U\|_{2\rightarrow\infty}\le \sqrt{r}\underset{i\in[2n],j\in[r]}{\textnormal{ max }}|\langle (\hat{\mathfrak{O}}-\mathfrak{O})u_{j},e_{i}\rangle|.
	\end{align*}
	For all $1 \le k \le n$, $U_{(k+n),j}=U_{k,j}$, and for each $1 \le i \le n$ and $1 \le j \le r$,
	\begin{align*}
		&\langle(\hat{\mathfrak{O}}-\mathfrak{O})u_{j},e_{i}\rangle
		= e_{i}^{\top}(\hat{\mathfrak{O}}-\mathfrak{O})u_{j}
		%\\
		%&= \sum_{k=1}^{n}(A_{i,k}^{1}-P_{i,k})U_{k,j}
		%+ \sum_{k=n+1}^{2n}\tfrac{1}{2}\left(A_{i,(k-n)}^{1}+A_{i,(k-n)}^{2}-2P_{i,(k-n)}\right)U_{k,j}\\
		%&
		= \sum_{k=1}^{n}\left(\tfrac{3}{2}A_{i,k}^{1}+\tfrac{1}{2}A_{i,k}^{2}-2P_{i,k}\right)U_{k,j}.
	\end{align*}
	Above, the roles of $A^{1}$ and $A^{2}$ are interchanged for $n+1 \le i \le 2n$.
	
	For any $1 \le i \le n$, the above expansion is a sum of independent (in $k$), bounded, mean zero random variables taking values in $[-2U_{k,j},2U_{k,j}]$.

	Hence by Hoeffding's inequality, with probability $1-o(1)$ in $n$,
	\begin{align*}
		\|(\hat{\mathfrak{O}}-\mathfrak{O})U\|_{2\rightarrow\infty}
		= \mathcal{O}_{r}(\log n).
	\end{align*}
	Similarly, for the matrix $U^{\top}(\hat{\mathfrak{O}}-\mathfrak{O})U\in\R^{r \times r}$,
	\begin{align*}
		\|U^{\top}(\hat{\mathfrak{O}}-\mathfrak{O})U\|_{2}
		&\le r \underset{i\in[r],j\in[r]}{\textnormal{ max }}|\langle (\hat{\mathfrak{O}}-\mathfrak{O})u_{j},u_{i}\rangle|,
	\end{align*}
	so for $1 \le i,j \le r$, then
	\begin{align*}
		\langle(\hat{\mathfrak{O}}-\mathfrak{O})u_{j},u_{i}\rangle
		= u_{i}^{\top}(\hat{\mathfrak{O}}-\mathfrak{O})u_{j}
		&= \sum_{1 \le l < k \le n}4(A_{l,k}^{1}+A_{l,k}^{2}-2P_{l,k})U_{k,j}U_{l,i}.
	\end{align*}
	This sum decomposes as a sum of independent, mean zero, bounded random variables taking values in $[-8U_{k,j}U_{l,i},8U_{k,j}U_{l,i}]$. By another application of Hoeffding's inequality, with probability $1-o(1)$,
	\begin{align*}
	\|U^{\top}(\hat{\mathfrak{O}}-\mathfrak{O})U\|_{2}
	= \mathcal{O}_{r}(\log n).
	\end{align*}
	Lemma~\ref{lemma:geom_resid_bounds} bounds $\|U^{\top}\hat{U}-W_{U}\|_{2}$ by $\|\sin\Theta(\hat{U},U)\|_{2}^{2}$ which is $\mathcal{O}(1/\Delta)$. Cumulatively, this demonstrates that
	$
	\|\hat{U} - U W_{U}\|_{\ttinf}
	=\mathcal{O}_{r}\left((\log n)/\Delta\right)
	$
	with probability $1-o(1)$ as $n\rightarrow\infty$.
\end{proof}

\section*{Acknowledgments}
The authors thank the editors, the associate editor, and the referees for their helpful comments and suggestions which have helped improve the exposition and presentation of this paper.

\begin{comment}
\begin{supplement}[id=suppA]
	\sname{Supplement A}
	\stitle{Title}
	\slink[doi]{COMPLETED BY THE TYPESETTER}
	\sdatatype{.pdf}
	\sdescription{Some text}
\end{supplement}
\end{comment}

\bibliographystyle{amsplain}
\bibliography{2ToInf}

\end{document}